\documentclass[final,onefignum,onetabnum]{siamart220329}

\usepackage{braket,amsfonts}
\usepackage{amsmath,amssymb}
\usepackage{mathtools}
\usepackage{mathrsfs}
\allowdisplaybreaks 
\usepackage{array}
\usepackage{bm}
\usepackage{makecell}
\usepackage{hhline}
\usepackage{nicematrix}

\usepackage[caption=false]{subfig}
\usepackage{pgfplots}

\newsiamremark{hypothesis}{Hypothesis}
\newsiamremark{remark}{Remark}

\usepackage{algorithmic}

\usepackage{graphicx,epstopdf}

\Crefname{ALC@unique}{Line}{Lines}

\usepackage{amsopn}

\usepackage{xspace}
\usepackage{bold-extra}
\usepackage[most]{tcolorbox}
\usepackage{enumerate}

\def\d{\mathrm{d}}

\def\p{\partial}

\def\var{\mathbf{Var}}

\def\tilde{\widetilde}
\def\hat{\widehat}
\def\bar{\overline}
\def\ubar{\underline}
\def\mb{\mathbb}
\def\mbf{\mathbf}
\def\mc{\mathcal}

\def\mr{\mathrm}
\def\msc{\mathscr}
\def\bra{\left\langle}
\def\ket{\right\rangle}
\def\ra{\rightarrow}

\def\epsilon{\varepsilon}

\def\Rd{\mathbb{R}^{d}}

\crefname{hypothesis}{Hypothesis}{Hypotheses} 
\newsiamremark{example}{Example} 
\crefname{example}{Example}{Example} 

\colorlet{texcscolor}{blue!50!black}
\colorlet{texemcolor}{red!70!black}
\colorlet{texpreamble}{red!70!black}
\colorlet{codebackground}{black!25!white!25}

\lstdefinestyle{siamlatex}{%
  style=tcblatex,
  texcsstyle=*\color{texcscolor},
  texcsstyle=[2]\color{texemcolor},
  keywordstyle=[2]\color{texemcolor},
  moretexcs={cref,Cref,maketitle,mathcal,text,headers,email,url},
}

\tcbset{%
  colframe=black!75!white!75,
  coltitle=white,
  colback=codebackground, 
  colbacklower=white, 
  fonttitle=\bfseries,
  arc=0pt,outer arc=0pt,
  top=1pt,bottom=1pt,left=1mm,right=1mm,middle=1mm,boxsep=1mm,
  leftrule=0.3mm,rightrule=0.3mm,toprule=0.3mm,bottomrule=0.3mm,
  listing options={style=siamlatex}
}

\DeclareTotalTCBox{\code}{ v O{} }
{ 
  fontupper=\ttfamily\color{black},
  nobeforeafter,
  tcbox raise base,
  colback=codebackground,colframe=white,
  top=0pt,bottom=0pt,left=0mm,right=0mm,
  leftrule=0pt,rightrule=0pt,toprule=0mm,bottomrule=0mm,
  boxsep=0.5mm,
  #2}{#1}

\patchcmd\newpage{\vfil}{}{}{}
\title{Hamilton-Jacobi-Bellman equations in the Wasserstein space
for the optimal control of the Kushner-Stratonovich equation\thanks{
This version was revised on February 28, 2025.
\funding{
The work was supported in part by the National Key Research
and Development Program of China under Grant 2022YFA1006102, 
and in part by the NSF of China under Grant 12471418.}}}

\author{
Hexiang Wan\thanks{School of Control Science and Engineering,
Shandong University,
Jinan, Shandong 250061, People's Republic of China
(\email{math.hilbertwan@gmail.com}).} 
\and Jie Xiong\thanks{
Department of Mathematics
and SUSTech International Center for Mathematics,
Southern University of Science and Technology,
Shenzhen, Guangdong 518055, People's Republic of China
(\email{xiongj@sustech.edu.cn}).}
}

\headers{Guide to Using  SIAM'S \LaTeX\ Style}
\ifpdf
\hypersetup{ pdftitle={Guide to Using  SIAM'S \LaTeX\ Style} }
\fi

\begin{document}
\maketitle

\begin{abstract}
This paper develops a comparison theorem for viscosity solutions of
a new class of Hamilton-Jacobi-Bellman (HJB) equations,
which is used to solve the separated problem governed by
the K-S equation in the Wasserstein space.
A distinctive feature of these HJB equations is
the simultaneous presence of variational and Lions derivatives,
an inevitable consequence of a nonzero observation function.
Moreover, the presence of state-dependent correlated noise adds
further complexity to the analysis,
making the proof of the comparison theorem more challenging.
The core proof strategy for the comparison theorem is to introduce
a novel adaptation of the doubling variables argument,
specifically tailored to tackle the challenges posed by the Wasserstein space.
To this end, we construct an entirely new bivariate functional
that combines viscosity sub/supsolutions,
a smooth gauge-type function, and two Gaussian regularized entropy penalizations.
The gauge-type function compensates for the non-differentiability
(in the Lions sense) of the 2-Wasserstein distance,
while the entropy penalizations ensure that the maximal point of the functional is well-behaved.
Another major contribution of this paper is the derivation of
the second-order variational and Lions derivatives of
the gauge-type function and entropy functional,
which are crucial for dealing with the second-order HJB equation.
Meanwhile, the simple structure and pleasing regularity of these derivatives
make the methodologies developed in this paper applicable
within a wider range of theoretical settings.
This paper gives the first result concerning the uniqueness of viscosity solutions
for second-order HJB equations with variational derivatives in the Wasserstein space.
\end{abstract}



\section{Introduction} 

\subsection{Main result}

This paper is devoted to investigating the \emph{comparison theorem} (see \Cref{thm:comparison})
for viscosity solutions of a new class of second-order elliptic HJB equations,
which correspond to the optimal control problem driven by
the Kushner-Stratonovich (K-S) equation.
Such HJB equations have the following form
\begin{align}
    &\ u(\mu) - \inf_{\gamma\in\Gamma} \left\{
    L(\mu,\gamma) + \msc{A}^{\gamma}u(\mu) \right\} \notag \\
    =&\ u(\mu) - \mu\otimes\mu \bra h - \mu(h),\sigma^{2,\top}\partial_{x}
    \frac{\delta^{2}u(\mu)}{\delta\mu^{2}} \ket_{\mb{R}^{q}} \notag \\
    &\ - \frac{1}{2} \mu\otimes\mu \left\{
    \bra h - \mu(h),h - \mu(h) \ket_{\mb{R}^{q}} \frac{\delta^{2}u(\mu)}{\delta\mu^{2}}
    + \mr{Trace} \left[ \sigma^{2} \sigma^{2,\top}
    \partial_{\mu\mu}^{2}u(\mu) \right] \right\} \label{eq:HJB} \\
    &\ - \inf_{\gamma\in\Gamma}\left\{L(\mu,\gamma) +
    \mu\left( \bra b,\partial_{\mu}u(\mu) \ket_{\mb{R}^{d}}
    + \frac{1}{2}\mr{Trace} \left[ a^{}\partial_{x}\partial_{\mu}u(\mu) \right] \right)
    \right\} \notag \\
    =&\ 0~~\text{in}~~\mc{P}_{}(\Rd), \notag
\end{align}
where $\mc{P}(\mb{R}^{d})$ is the space of Borel probability measures on $\Rd$,
and $\msc{A}^{\gamma}$ is a specific degenerate
second-order differential operator on $\mc{P}(\Rd)$.
The HJB equation incorporates both \emph{variational derivatives} (cf. \Cref{def:variational})
and \emph{Lions derivatives} (cf. \Cref{def:Lions}),
as shown in \Cref{eq:HJB}.
This indicates that the well-posedness results from previous work,
which only relied on Lions derivatives, are inapplicable to this study.
Simultaneously, we permit the observation function $h(\cdot)\neq0$
in the state-observation pair model (see \Cref{eq:state-observation})
and allow the correlated noise to be state-dependent,
which essentially causes the appearance of the variational derivatives
$\p_{x}\frac{\delta^{2}u(\mu)}{\delta\mu^{2}}$ and $\frac{\delta^{2}u(\mu)}{\delta\mu^{2}}$
in \Cref{eq:HJB}.

The primary motivation for studying \Cref{eq:HJB} in $\mc{P}(\Rd)$ is
the partially observed stochastic optimal control problem with \emph{non-smooth coefficients}.
A standard strategy is to transform the original problem into a \emph{separated problem},
in which case the K-S equation or Zakai equation serves as a new control system.
Nevertheless, despite its infinite-dimensional state space,
the separated problem is an optimal control problem with complete observation.
Moreover, The separated problem described by the K-S equation investigated in this study
is particularly well-suited for scenarios wherein
\emph{the density of the optimal filter either does not exist
or exhibits insufficient regularity}.

\subsection{Previous results}

Over the past few decades, there has been an ongoing investigation of
partial differential equations (PDEs) on infinite-dimensional spaces.
The theory of first- and second-order PDEs in Hilbert spaces
has been thoroughly investigated.
The monograph \cite{FGS2017} details the fruitful results to date
in the theory of HJB equations in the Hilbert space framework.
The separated problem modeled by the Zakai equation in this setting has also been satisfactorily resolved
(see, e.g., \cite{GS2000,Lions1989}). 

The initial interest in PDEs in $\mc{P}(\Rd)$ stems from the separated problem
created by the study of partially observed optimal control problems.
However, in this situation,
there are relatively few studies that deal with the separated problem.
For instance, early literature \cite{Fleming1982} merely validated
the dynamic programming principle in $\mc{P}(\Rd)$,
while \cite{Hijab1990} established the existence of viscosity solutions for
the HJB equation in $\mc{P}(\Rd)$ but failed to obtain uniqueness.
A renewed interest in these PDEs in $\mc{P}(\Rd)$ has emerged in recent years
as a result of their connection with mean field control and mean field game problems.
In turn, recent research in this field has provided us with new tools to
resolve the separated problem.

Fortunately, the study of viscosity solutions for HJB equations in the Wasserstein space
has made substantial advances in recent years
owing to the accelerated development of mean field-related theory.
We do not intend to assemble a full list of pertinent literature;
rather, we will provide a concise overview of several strategies
for proving the uniqueness of viscosity solutions in the Wasserstein space
that are relevant to the topic of this paper.
A few years ago, Bandini et al. \cite{BCFP2019} proposed the control randomization method
and obtained a new HJB equation in the Wasserstein space for the separated problem
driven by the Zakai equation.
They adopted the \emph{lifting method} to transform the HJB equation into Hilbert space
and subsequently made use of the existing result in Hilbert space (i.e., \cite{GS2000})
to illustrate the uniqueness of viscosity solutions. 
If the correlated noise in the state-observation pair model is state-independent,
the second-order Lions derivative in the HJB equation of the separated problem
governed by the Zakai equation
is analogous to the finite-dimensional Laplacian operator,
according to the authors of \cite{BEZ2023}.
They utilized the Borwein-Preiss variational principle to overcome the difficulty of
the Wasserstein space not being locally compact,
then combined it with the Crandall-Ishii lemma to prove the comparison theorem.
The approach they employed is known as \emph{doubling variables with weaker metrics}.
Such an approach is also applicable to complete observation scenarios
(see, for instance, \cite{DJS2025,SY2024}).
We point out that both papers \cite{BCFP2019,BEZ2023}
address the case where the observation function $h(\cdot)\equiv0$.

The third way to demonstrate the uniqueness of viscosity solutions
for HJB equations in the Wasserstein space is to use the method of
\emph{approximating the value function with smooth functions} (cf. \cite{CTQ2025,CGKPR2024}).
This method also necessitates the Borwein-Preiss variational principle
when handling the problem in the Wasserstein space.
It is worth noting that such a variational principle is also critical
in the setting of viscosity solutions for path-dependent HJB equations (cf. \cite{JJZhou2023}).
Finally, we present the recently developed fourth method:
\emph{doubling variables with entropy penalization} (see \cite{DS2024}).
It ought to point out that Feng and Kurtz
exploited the entropy penalization technique to explore large deviations
and comparison theorems for various kinds of first-order Hamilton-Jacobi equations;
see their monograph \cite{FK2006} and the literature therein.
However, as indicated in \cite{DJS2025},
the fourth methodology is currently unable to address the second-order HJB equation,
particularly in the presence of correlated noise.
To the authors' knowledge, the third and fourth methods
have not yet been used to deal with the separated problem.

By restricting the scope of the research to partially observed stochastic systems
and selecting $\mc{P}(\mb{R}^{d})$ or the Wasserstein space as the state space,
we compare this paper with multiple papers in five aspects,
as illustrated in the table below.

\begin{table}[tbhp]
\footnotesize
\caption{Comparisons between our research and several existing pertinent articles.}
\label{table:comparison}
\begin{center}
\begin{tabular}{|c||c|c|c|c|c|}
\hline \makecell{\textbf{References} \\ \textbf{(sort by}\\ \textbf{time} \\ \textbf{from top}
       \\ \textbf{to bottom)}}
       & \makecell{\textbf{The}\\ \textbf{system} \\ \textbf{equation} \\
       \textbf{for the} \\ \textbf{separated} \\ \textbf{problem}}
       & \makecell{\textbf{Is the} \\ \textbf{observation} \\ \textbf{function}
       \\ \textbf{$h(\cdot)\neq0$?}}
       & \makecell{\textbf{Is the} \\ \textbf{correlated} \\ \textbf{noise $\sigma^{2}(\cdot)$}
       \\ \textbf{state-} \\ \textbf{dependent?}}
       & \makecell{\textbf{Does the} \\ \textbf{HJB} \\ \textbf{equation} \\ \textbf{contain}
       \\ \textbf{variational} \\ \textbf{derivatives?}}
       & \makecell{\textbf{Has the} \\ \textbf{uniqueness}
       \\ \textbf{of viscosity} \\ \textbf{solutions} \\ \textbf{been} \\ \textbf{proven?}} \\
\hline Hijab \cite{Hijab1990} & K-S eq. & Yes & No & No & No \\
\hline \makecell{Bandini \\ et al. \cite{BCFP2019}}
       & Zakai eq. & No & Yes & No & Yes \\
\hline \makecell{Bayraktar \\ et al. \cite{BEZ2023}} & Zakai eq. & No & No & No & Yes \\
\hline Our paper & K-S eq. & Yes & Yes & Yes & Yes \\
\hline
\end{tabular}
\end{center}
\end{table}

\subsection{Comments on the proof}

We reiterate that this paper discusses
a large class of separated problems characterized
by the K-S equation in the Wasserstein space,
appropriate for situations where the probability density
is either nonexistent or its regularity is insufficient.

In this paper, we adopt the method of doubling variables with entropy penalization to
establish a comparison theorem for viscosity solutions of \Cref{eq:HJB}.
One of the most important technical parts of the proof
is demonstrating the regularity properties
of the maximum point of the following bivariate functional
\begin{equation*}
 \begin{aligned}
  \mc{P}_{2}(\mb{R}^{d})\times\mc{P}_{2}(\mb{R}^{d})\ni
  (\mu,\nu)\mapsto u^{1}(\mu) - u^{2}(\nu) - \frac{\alpha}{2}\mc{G}(\mu,\nu)
  - \beta\left\{ \mc{E}(\mu) + \mc{E}(\nu) \right\}.
 \end{aligned}
\end{equation*}
Here, $u^{1}$ and $u^{2}$ represent the sub- and super-solutions, respectively.
$\mc{G}$ is a gauge-type function (cf. \Cref{def:gauge}),
$\mc{E}(\mu)=\int_{\mb{R}^{d}} \log(\mu)\d\mu$
represents the entropy of $\mu\in\mc{P}_{\mr{ac},2}(\mb{R}^{d})$ (cf. \Cref{def:entropy}),
and $\alpha,\beta>0$ are small parameters.
See \Cref{thm:max} for the proof of the above result,
which employs the fact that any bounded sequences in 2-Wasserstein space
are sequentially compact using the 1-Wasserstein topology.
Actually, we have mentioned previously a different approach for dealing with
the absence of local compactness in the Wasserstein space:
utilizing a gauge-type function and the Borwein-Preiss variational principle
to produce maximum/minimum.
We emphasize that the former method is more appropriate
for the entropy penalization scenario.

As a prerequisite for our successful treatment of
the variational derivatives in \Cref{eq:HJB},
we must calculate the second-order variational derivatives
and Lions derivatives of the gauge-type function and entropy functional.
It is worthwhile to point out that the second-order Lions derivative of
the gauge-type function constructed in \cite{CGKPR2024} is unbounded (see \Cref{rmk:gauge}).
As a result, regardless of whether the HJB equation contains variational derivatives,
the gauge-type function in \cite{CGKPR2024} cannot handle the general second-order HJB equation
in the Wasserstein space.
Therefore, in this paper, we modify the gauge-type function in \cite{CGKPR2024}
to guarantee the boundedness of its second-order Lions derivative.
Finally, we highlight that the entropy functional necessitates Gaussian regularization,
which is vital in calculating its second-order variational derivative
(see \Cref{rmk:Gauss}).

The paper is structured as follows.
In \Cref{sec:separated}, we formulate our
separated problem driven by the K-S equation
and introduce some basic notation and assumptions.
\Cref{sec:viscosity} introduces the concepts of variational derivatives,
Lions derivatives, dynamic programming principles,
HJB equations, and viscosity solutions.
In \Cref{sec:entropy}, we provide a comprehensive investigation of
the gauge-type function, Gaussian regularized entropy penalization, and their derivatives,
and we finalize this section with a maximum principle.
\Cref{sec:comparison} is devoted to proving the comparison theorem,
which implies the uniqueness of viscosity solutions.
In \Cref{sec:app}, we define weakly admissible control strategies
and present the separated problem
for a broad class of finite-dimensional state-observation pairs.
Furthermore, we offer a succinct exploration of potential applications of
the separated problem within the framework of state-observation pairs.
The paper concludes with a few comments in the final section.

\section{The separated problem formulated by the K-S equation}\label{sec:separated}

At the beginning of this section,
we list the main symbols used throughout the paper,
as well as the definitions of the basic function spaces.
In applications to physics and engineering,
\emph{test functions} are usually infinitely differentiable
real-valued functions with compact support that are defined on $\Rd$. 
The set of all such test functions forms a vector space that is denoted by
$C_{0}^{\infty}(\Rd)$ or $\mc{D}(\Rd)$.
For instance, bump functions are test functions in $\mc{D}(\Rd)$,
whereas Gaussian functions are test functions
in the space of rapidly decreasing functions on $\Rd$
without compact support (Schwartz space).
A \emph{distribution} on $\Rd$
is an element of the continuous dual space of $\mc{D}(\Rd)$
when $\mc{D}(\Rd)$ is endowed with the locally convex topology.
The space of all distributions on $\Rd$ is usually denoted by $\mc{D}^{\prime}(\Rd)$.
The angle brackets $\bra \mc{T},\phi \ket:=\mc{T}(\phi)$ represent the canonical duality product
between the distribution $\mc{T}$ on $\Rd$ and the test function $\phi\in\mc{D}(\Rd)$.

Let $k\in\mb{N}$ and $p\in[1,\infty]$.
Let $C_{b}^{k}(\Rd)$ denote the vector space of all $k$-times
bounded continuously differentiable real-valued functions on $\Rd$.
In the distributional sense,
the Sobolev space $W^{k,p}(\Rd)$ is defined as the collection of all functions on $\Rd$
with derivatives up to order $k$ belonging to $L^{p}(\Rd)$.
We will also use the notation $H^{k}(\Rd)\triangleq W^{k,2}(\Rd)$.
The space $L_{\mr{loc}}^{p}(\Rd)$ consists of locally $L^{p}$-integrable functions,
which are $L^{p}$-integrable on every compact subset of its domain of definition.
For $\mu\in\mc{P}(\Rd)$, we define $L^{p}(\mu;\Rd)$ as the set of all
$\Rd$-valued functions that satisfy
\begin{align*}
  \Vert \phi \Vert_{L^{p}(\mu;\Rd)} := \left( \int_{\Rd}
  \Vert \phi(x) \Vert_{\Rd}^{p} \d \mu(x)
  \right)^{\frac{1}{p}} < \infty.
\end{align*}

Let us recall that a Borel probability measure $\mu$ in $\Rd$ belongs to $\mc{P}_{p}(\Rd)$
if  the $p$-moment $\int_{\mb{R}^{d}} |x|^{p} \d \mu(x)$ is finite.
The $p$-Wasserstein distance between $\mu,\nu\in\mc{P}_{p}(\Rd)$
is defined by
\begin{equation*}
 \begin{aligned}
   \mc{W}_{p}^{}(\mu,\nu) := \inf_{\pi\in\Pi(\mu,\nu)} \left(
   \iint_{\Rd\times\Rd} |x-y|^{p} \d \pi(x,y) \right)^{\frac{1}{p}},
 \end{aligned}
\end{equation*}
where $\Pi(\mu,\nu)\subset\mc{P}(\mb{R}^{d}\times\mb{R}^{d})$
denotes the set of Borel probability measures $\pi$ with marginal distributions $\mu$ and $\nu$.
It turns out that $(\mc{P}_{p}(\Rd),\mc{W}_{p})$
is a complete and separable metric space (cf. \cite[Lecture 8]{ABS2021}).
The space of absolutely continuous probability measures
with finite $p$-moment will be denoted by
\begin{equation*}
 \begin{aligned}
   \mc{P}_{\mr{ac},p}(\mb{R}^{d}) := \left\{ \mu\in\mc{P}_{p}(\mb{R}^{d}):
   \d \mu(x) = \mu(x) \d x~\text{for some $\mu(x)\in L^{1}(\mb{R}^{d})$} \right\}.
 \end{aligned}
\end{equation*}
It is easy to show that $\mc{P}_{\mr{ac},p}(\Rd)$ is dense in $\mc{P}_{p}(\Rd)$
when the distance $\mc{W}_{p}$ is taken.
As soon as a functional $\digamma:\mc{P}_{p}(\Rd)\ra\mb{R}^{1}$ is Lipschitz continuous
with respect to the distance $\mc{W}_{p}$,
its Lipschitz constant is denoted by
\begin{align*}
  \mr{Lip}(\digamma;\mc{W}_{p}) := \sup_{\mu,\nu\in\mc{P}_{p}(\Rd),~\mu\neq\nu}
  \frac{\left|\digamma(\mu) - \digamma(\nu) \right|}{\mc{W}_{p}(\mu,\nu)}.
\end{align*}

Let $(\Omega,\mc{F},\mb{P})$ be a complete probability space
with a filtration $(\mc{F}_{t})_{t\geq0}$ that satisfies the usual conditions.
On $(\Omega,\mc{F},\mb{P})$,
we consider a $q$-dimensional Wiener process $\beta$.
Let the control region $\Gamma\subset\mb{R}^{m}$ be convex and compact.
The $\mc{P}(\mb{R}^{d})$-valued process $\pi_{t}$ satisfies the
K-S equation
\begin{equation}\label{eq:Kushner}
 \begin{aligned}
  \pi_{t}(\phi) = \pi_{0}(\phi)
  + \int_{0}^{t} \pi_{r}(\msc{L}^{\gamma}\phi) \d r
  + \int_{0}^{t}
  \bra \pi_{r}(h\phi + \msc{M}^{}\phi) - \pi_{r}(h)\pi_{r}(\phi),\d \beta_{r} \ket_{\mb{R}^{q}},
 \end{aligned}
\end{equation}
where $\phi\in\mc{D}(\Rd)$, $\pi_{0}\sim\mu\in\mc{P}(\mb{R}^{d})$, and $\pi_{0}(\phi)\triangleq\mu(\phi)$.
The differential operators $\msc{L}^{\gamma}$ and $\msc{M}^{}$ are defined as follows
\begin{align*}
  (\msc{L}^{\gamma}\phi)(x) =&\ \frac{1}{2}
  \sum_{i,j=1}^{d} \left(\sigma^{1}\sigma^{1,\top}
  + \sigma^{2}\sigma^{2,\top}\right)_{ij}(x,\gamma)\partial_{ij}^{2}\phi(x)
  + \sum_{i=1}^{d} b_{i}(x,\gamma)\partial_{i}\phi(x) \\
  \triangleq&\ \frac{1}{2}
  \sum_{i,j=1}^{d} a_{ij}^{}(x,\gamma)\partial_{ij}^{2}\phi(x)
  + \sum_{i=1}^{d} b_{i}(x,\gamma)\partial_{i}\phi(x), \\
  (\msc{M}^{}\phi)(x) =&\
  \sigma^{2,\top}(x) \partial_{x}\phi(x),
\end{align*}
where
$\sigma^{1}(\cdot,\cdot):\mb{R}^{d}\times\Gamma\ra\mb{R}^{d\times d}$,
$\sigma^{2}(\cdot):\mb{R}^{d}\ra\mb{R}^{d\times q}$,
$b(\cdot,\cdot):\mb{R}^{d}\times\Gamma\ra\mb{R}^{d}$,
and $h(\cdot):\mb{R}^{d}\ra\mb{R}^{q}$
are all deterministic measurable functions.
Here, the operator $\msc{M}$ represents the impact caused by correlated noise.
Typically we set $\mc{D}(\msc{L}^{\gamma})=C_{b}^{2}(\mb{R}^{d})$,
$\mc{D}(\msc{M}^{})=C_{b}^{1}(\mb{R}^{d})$.
Let $\mc{F}_{t}^{\beta}=\sigma\{ \beta_{r},0\leq r\leq t \}$
symbolize the filtration generated by $\beta$.
In this model, we assume that the admissible control process
is $\mc{F}_{t}^{\beta}$-progressively measurable with values in $\Gamma$.
We denote by $\bm{\Gamma}$ the set of admissible control processes.

Now, let us make necessary assumptions with regard to the above model.

\begin{hypothesis}\label{hyp:BC}
The functions $b$, $\sigma^{1}$, $\sigma^{2}$ and $h$
are bounded and Lipschitz continuous functions, i.e.,
there exists a positive constant $C_{}$ such that
for all $x,y\in\Rd, \gamma\in\Gamma$,
\begin{align*}
  &\ |b(x,\gamma)| + |\sigma^{1}(x,\gamma)| + |\sigma^{2}(x)| + |h(x)| \leq C_{}, \\
  &\ |b(x,\gamma)-b(y,\gamma)| + |\sigma^{1}(x,\gamma)-\sigma^{1}(y,\gamma)|
  + |\sigma^{2}(x)-\sigma^{2}(y)| + |h(x)-h(y)| \leq C_{}|x-y|.
\end{align*}
\end{hypothesis}

We apply the \emph{particle system approximation method} from Kurtz and Xiong \cite{KX1999}
to establish the well-posedness of \Cref{eq:Kushner}
under the assumptions of \Cref{hyp:BC},
which are sufficient for this purpose.
In this context, \emph{well-posedness} refers to the existence and uniqueness of a strong solution
in the filtered probability space $(\Omega,\mc{F},\{\mc{F}_{t}^{}\}_{t\geq 0},\mb{P})$.

The aim of our optimal control problem is to minimise,
over all admissible control processes $\gamma_{\cdot}\in\bm{\Gamma}$, the criterion
\begin{align}\label{eq:cost-1}
  \mc{J}(\mu;\gamma_{\cdot}) =&\ \mb{E}\left[ 
  \int_{0}^{\infty} e^{-t} L(\pi_{t},\gamma_{t}) \d t \right],
\end{align}
where $\mu=\pi_0$ represents the initial distribution of $\pi_{t}$,
and $L(\cdot,\cdot):\mc{P}(\Rd)\times\Gamma\ra\mb{R}^{1}$
is a deterministic measurable function that satisfies the following condition.

\begin{hypothesis}\label{hyp:Lip1}
$L$ is Lipschitz continuous in $\mu\in\mc{P}(\Rd)$ with respect to the $\mc{W}_{1}$ topology,
uniformly in $\gamma\in\Gamma$.
\end{hypothesis}

We define the \emph{value function} for the optimal control problem as follows
\begin{equation}\label{eq:valuef}
 \begin{aligned}
  u(\mu) := \inf_{\gamma_{\cdot}\in\bm{\Gamma}} \mc{J}(\mu;\gamma_{\cdot})
  = \inf_{\gamma_{\cdot}\in\bm{\Gamma}} \mb{E} \left[ \int_{0}^{\infty}
  e^{-t} L(\pi_{t},\gamma_{t}) \d t \right].
 \end{aligned}
\end{equation}
The paper characterizes the value function $u$ via viscosity solutions.

\section{Viscosity solutions of the HJB equation}
\label{sec:viscosity}

We step through the mathematical tools necessary to define viscosity solutions.

\begin{definition}[Variational Derivative]\label{def:variational}
A functional $\upsilon:\mc{P}_{2}(\mb{R}^{d})\ra\mb{R}^{1}$ is said to have
a variational derivative if there is a map 
$\frac{\delta\upsilon}{\delta\mu}:\mc{P}_{2}(\mb{R}^{d})\times\mb{R}^{d}\ra\mb{R}^{1}$
such that for every $\mu,\nu\in\mc{P}_{2}(\mb{R}^{d})$, it holds that
\begin{equation*}
 \begin{aligned}
  \upsilon(\mu) - \upsilon(\nu) = \int_{0}^{1} \int_{\mb{R}^{d}}
  \frac{\delta\upsilon}{\delta\mu}\left(t\mu + (1-t)\nu,x \right)
  \d (\mu - \nu)(x) \d t.
 \end{aligned}
\end{equation*}
In the same way, we can continue to define higher-order variational derivatives.
The class of functions that are differentiable of order $k$
in the variational sense
is denoted by $\mc{C}^{k}(\mc{P}_{2}(\Rd))$.
\end{definition}

\begin{definition}[Lions Derivative]\label{def:Lions}
A functional $\upsilon\in\mc{C}^{1}(\mc{P}_{2}(\mb{R}^{d}))$ is said to be
first-order L-differentiable if the map
$x\mapsto\frac{\delta\upsilon}{\delta\mu}(\mu,x)$
is differentiable. We set
\begin{equation*}
 \begin{aligned}
  \partial_{\mu}\upsilon(\mu,x):=
  \partial_{x}\frac{\delta\upsilon}{\delta\mu}(\mu,x)\in\mb{R}^{d}.
 \end{aligned}
\end{equation*}
$\mc{C}_{L}^{1}(\mc{P}_{2}(\mb{R}^{d}))$ denotes this class of functions.
Similarly, if the map $x\mapsto\partial_{\mu}\upsilon(\mu,x)$ is twice differentiable,
we also define
\begin{equation*}
 \begin{aligned}
  \partial_{\mu\mu}^{2}\upsilon(\mu,x,y)\triangleq\partial_{\mu}(\partial_{\mu}\upsilon(\mu,x))(y)
  := \partial_{x}\partial_{y}^{\top} \frac{\delta^{2}\upsilon}{\delta\mu^{2}}(\mu,x,y).
 \end{aligned}
\end{equation*}
We denote by $\mc{C}_{L}^{2}(\mc{P}_{2}(\mb{R}^{d}))$
the set of this class of functions.
\end{definition}
Further information about variational derivatives, Lions derivatives,
and their connections can be found in \cite[Section 5.4.1]{CD2018-1}
and \cite[Section 2]{Martini2023}.


Let $\phi^{1},\phi^{2}\in C_{b}(\mb{R}^{d})$
and $\phi^{3}\in C_{b}(\mb{R}^{d}\times\mb{R}^{d})$.
We adopt the following notations
\begin{align*}
  (\mu\otimes\mu)(\phi^{1}\phi^{2}) := &\ \iint_{\mb{R}^{2d}}
  \phi^{1}(x)\phi^{2}(y) \d\mu(x) \d\mu(y), \\
  (\mu\otimes\mu)(\phi^{1}\phi^{2}\phi^{3}) := &\ \iint_{\mb{R}^{2d}}
  \phi^{1}(x)\phi^{2}(y)\phi^{3}(x,y) \d\mu(x)\d\mu(y).
\end{align*}
By engaging these two kinds of derivatives,
we display a new result of the generalization of the It\^{o} formula
in the space of probability measures.

\begin{lemma}[It\^{o} Formula]
Let $\pi_{t}$ be a solution to the K-S equation
and $\upsilon\in\mc{C}_{L}^{2}(\mc{P}_{2}(\mb{R}^{d}))$.
Then it holds that
\begin{align*} 
  \d \upsilon(\pi_{t}) =&\ \pi_{t}\left( \msc{L}^{\gamma}
  \frac{\delta\upsilon(\pi_{t})}{\delta\mu} \right) \d t
  + \bra \pi_{t}\left[ \left(h + \msc{M}^{} - \pi_{t}(h)\right)
  \frac{\delta\upsilon(\pi_{t})}{\delta\mu} \right], \d \beta_{t} \ket_{\mb{R}^{q}} \\
  &\ + \frac{1}{2} \pi_{t}\otimes\pi_{t} \left[
  \bra h + \msc{M}^{} - \pi_{t}(h),h + \msc{M}^{} - \pi_{t}(h) \ket_{\mb{R}^{q}}
  \frac{\delta^{2}\upsilon(\pi_{t})}{\delta\mu^{2}} \right] \d t.
 \end{align*}
\end{lemma}
Proof of this lemma is available in \cite{Martini2023}.
It is important to remember that the It\^{o} formula
encompasses both linear variational derivatives and Lions derivatives.
Therefore, the HJB equation must also incorporate these two kinds of derivatives.
This raises challenges for our subsequent investigation of
viscosity solutions of the HJB equation.


Now, we leverage the It\^{o} formula to characterize
the operator $\msc{A}^{\gamma}$ appearing in \Cref{eq:HJB}.
Let $\mc{D}(\msc{A}^{\gamma})$ be the set of all
$\upsilon\in\mc{C}_{L}^{2}(\mc{P}_{2}(\mb{R}^{d}))$ of cylindrical functions
$\upsilon(\mu) = \Upsilon(\mu(\phi^{1}),...,\mu(\phi^{n}))$
for some $n\in\mb{N}\backslash\{0\}$,
where $\Upsilon\in C_{b}^{\infty}(\mb{R}^{n})$
,$\phi^{1},...,\phi^{n}\in\mc{D}(\mb{R}^{d})$
and $\mu(\phi)$ denotes the integral of $\phi$ against the measure $\mu$ over $\mb{R}^{d}$.
For $\upsilon\in\mc{D}(\msc{A}^{\gamma})$, it holds that
\begin{align*} 
   (\msc{A}^{\gamma}\upsilon)(\mu)
   =&\ \frac{1}{2} \sum_{i,j=1}^{n}
   \partial_{ij}^{2} \Upsilon(\mu(\phi^{1}),...,\mu(\phi^{n}))
   \left[ \mu(h\phi^{i} + \msc{M}^{}\phi^{i}) - \mu(h)\mu(\phi^{i}) \right] \\
   &\ \cdot \left[ \mu(h\phi^{j} + \msc{M}^{}\phi^{j}) - \mu(h)\mu(\phi^{j}) \right]
   + \sum_{i=1}^{n}
   \partial_{i} \Upsilon(\mu(\phi^{1}),...,\mu(\phi^{n}))\mu(\msc{L}^{\gamma}\phi^{i}) \\
   = &\ \frac{1}{2} \mu\otimes\mu\left(
   \bra h + \msc{M}^{} - \mu(h),h + \msc{M}^{} - \mu(h) \ket_{\mb{R}^{q}}
   \frac{\delta^{2}\upsilon(\mu)}{\delta\mu^{2}} \right)
   + \mu\left( \msc{L}^{\gamma}\frac{\delta\upsilon(\mu)}{\delta\mu} \right).
\end{align*}
Please refer to \cite{Fleming1982}
for some extra descriptions of the operator $\msc{A}^{\gamma}$
and its corresponding nonlinear semigroup.

We state the dynamic programming principle (DPP)
without proof in the following theorem.
The DPP has been demonstrated in an extremely broad setting by \cite{DPT2022}
thanks to the measurable selection theorems 
(the proof strategy employed there is also applicable to this paper).

\begin{theorem}[DPP, Infinite Horizon Case]
Suppose that \Cref{hyp:BC,hyp:Lip1} hold.
For every $\tau>0$ and $\mu\in\mc{P}_{2}(\mb{R}^{d})$,
\begin{equation*}
 \begin{aligned}
   u(\mu) = \inf_{\gamma_{\cdot}\in\bm{\Gamma}} \mb{E}^{} \left[
   \int_{0}^{\tau} e^{-r} L(\pi_{r},\gamma_{r}) \d r
   + e^{-\tau} u(\pi_{\tau}) \right].
 \end{aligned}
\end{equation*}
\end{theorem}

The next result is the HJB equation,
which can be derived by combining the It\^{o} formula with DPP.

\begin{proposition}
Suppose that \Cref{hyp:BC,hyp:Lip1} hold.
Let the value function $u$ be in $\mc{C}_{L}^{2}(\mc{P}_{2}(\mb{R}^{d}))$.
Then $u$ satisfies the following second-order HJB equation:
\begin{equation}\label{eq:HJB-1}
 \begin{aligned}
    F \left(\mu,u,\partial_{\mu}u,\partial_{x}\partial_{\mu}u,\frac{\delta^{2}u}{\delta\mu^{2}},
    \partial_{x}\frac{\delta^{2}u}{\delta\mu^{2}},\partial_{\mu\mu}^{2}u\right)
    \triangleq &\ u(\mu) - \inf_{\gamma\in\Gamma} \left\{
    L(\mu,\gamma) + \msc{A}^{\gamma}u(\mu) \right\} \\
    =&\ 0~~\text{in}~~\mc{P}_{2}(\mb{R}^{d}).
 \end{aligned}
\end{equation}
\end{proposition}

At the end of this section, we give the definition of viscosity solutions.

\begin{definition}[Viscosity Solution]\label{def:viscosity}
A function $\upsilon:\mc{P}_{2}(\mb{R}^{d})\ra\mb{R}^{1}$ is
a viscosity solution of \Cref{eq:HJB-1} if the following assertions are satisfied:
\begin{enumerate}[\rm \bfseries(I)]
\item
$\upsilon$ is a viscosity subsolution of \Cref{eq:HJB-1},
that is, for any $\phi\in\mc{C}_{L}^{2}(\mc{P}_{2}(\mb{R}^{d}))$
and any local maximum point $\overline{\mu}\in\mc{P}_{2}(\mb{R}^{d})$ of $\upsilon-\phi$,
\begin{equation*}
 \begin{aligned}
   F \left(\bar{\mu},\upsilon(\bar{\mu}),\partial_{\mu}\phi(\bar{\mu}),
   \partial_{x}\partial_{\mu}\phi(\bar{\mu}),
   \frac{\delta^{2}\phi(\bar{\mu})}{\delta\mu^{2}},
   \partial_{x}\frac{\delta^{2}\phi(\bar{\mu})}{\delta\mu^{2}},
   \partial_{\mu\mu}^{2}\phi(\bar{\mu}) \right)
   \leq 0;
 \end{aligned}
\end{equation*}
\item
$\upsilon$ is a viscosity supsolution of \Cref{eq:HJB-1},
that is, for any $\phi\in\mc{C}_{L}^{2}(\mc{P}_{2}(\mb{R}^{d}))$
and any local minimum point $\underline{\mu}\in\mc{P}_{2}(\mb{R}^{d})$ of $\upsilon-\phi$,
\begin{equation*}
 \begin{aligned}
   F \left( \ubar{\mu},\upsilon(\ubar{\mu}),\partial_{\mu}\phi(\ubar{\mu}),
   \partial_{x}\partial_{\mu}\phi(\ubar{\mu}),
   \frac{\delta^{2}\phi(\ubar{\mu})}{\delta\mu^{2}},
   \partial_{x}\frac{\delta^{2}\phi(\ubar{\mu})}{\delta\mu^{2}},
   \partial_{\mu\mu}^{2}\phi(\ubar{\mu}) \right)
   \geq 0.
 \end{aligned}
\end{equation*}
\end{enumerate}
\end{definition}

\begin{remark}
The notion of viscosity solutions in the Wasserstein space
can be defined in various ways (sometimes not even equivalently).
For instance, the literature \cite{BCFP2019,CTQ2025,CGKPR2024,DS2024}
utilizes a family of \emph{smooth} test functions to define viscosity solutions
by \emph{touching} the value function from above and below,
while the literature \cite{BEZ2023,DJS2025} employs \emph{semi-jets} to define viscosity solutions.
Regardless, it is crucial that we identify the definition of viscosity solutions
in accordance with specific situations.
\end{remark}


\section{Entropy penalization}
\label{sec:entropy}

As in the previous section, we begin by introducing several critical tools.

\begin{definition}\label{def:entropy}
The (differential) \textbf{entropy} is defined as
\begin{equation*}
 \mc{E}(\mu) =
  \begin{dcases}
	\int_{\mb{R}^{d}} \log(\mu(x)) \d \mu(x) & \text{if $\mu\in\mc{P}_{\mr{ac},2}(\mb{R}^{d})$},\\
	\infty & \text{otherwise.}
  \end{dcases}
\end{equation*}
\end{definition}
It is also convenient to introduce the nonnegative version of the entropy
$\tilde{\mc{E}}(\mu):=\mc{E}(\mu) + \pi\int_{\mb{R}^{d}} |x|^2 \d \mu(x) \geq 0$.
Direct calculation shows the following inequality:
\begin{align}\label{eq:entropy-bound}
  \frac{\pi}{2} \int_{\Rd} |x|^{2}
  \d \mu(x) - \frac{d}{2}\log2
  \leq \tilde{\mc{E}}(\mu).
\end{align}

\begin{definition}
The \textbf{Fisher information} is defined as
\begin{equation*}
 \mc{I}(\mu) =
  \begin{dcases}
	4 \int_{\mb{R}^{d}} \left| \partial_{x} \sqrt{\mu}(x) \right|^{2} \d x
    & \text{if $\mu\in\mc{P}_{\mr{ac},2}(\mb{R}^{d})$ and $\sqrt{\mu}\in H^{1}(\Rd)$}, \\
	\infty & \text{otherwise.}
  \end{dcases}
\end{equation*}
\end{definition}

\begin{remark}
Fisher information can also be expressed in the following forms
\begin{align*}
 \mc{I}(\mu) =&\ 4 \int_{\mb{R}^{d}} \left| \partial_{x} \sqrt{\mu}(x) \right|^{2} \d x
  = \int_{\Rd} \left| \partial_{x}\log(\mu(x)) \right|^{2} \d \mu(x) \\
  =&\ \int_{\Rd} \frac{|\partial_{x}\mu(x)|^{2}}{\mu(x)} \d x
  = - \mr{Trace} \int_{\Rd} \partial_{xx}^{2}(\log(\mu(x))) \d \mu(x),
\end{align*}
with all of these identities being equivalent under appropriate regularity assumptions.
Note that $\mc{I}(\mu)$ is finite if and only if $\partial_{x}\log\mu(x)\in L^{2}(\mu;\Rd)$
(see \cite[Lemma D.7]{FK2006}).
\end{remark}

\begin{remark}\label{rmk:entropy-Fisher}
When selecting the Gaussian measure on $\Rd$ as the reference measure,
it is well-known that the \emph{Logarithmic-Sobolev inequality} builds a link
between entropy and Fisher information (cf. \cite[Proposition 15.23]{ABS2021}).
In this article, we will leverage the following fact (see \cite[Lemma 2.9]{DS2024}):
Assuming $\mu\in\mc{P}_{\mr{ac},2}(\Rd)$ and $\mc{I}(\mu)<\infty$,
then $|\mc{E}(\mu)|<\infty$.
\end{remark}

\begin{definition}
We denoted by $\mu*\nu$ the \textbf{convolution of measures}
$\mu,\nu\in\mc{P}_{2}(\mb{R}^{d})$ defined via
\begin{equation*}
 \begin{aligned}
   \int_{\mb{R}^{d}} \phi(x) \d (\mu*\nu)(x) :=
   \iint_{\mb{R}^{2d}} \phi(x+y)\d \mu(x) \d \nu(y)
 \end{aligned}
\end{equation*}
for all bounded measurable function $\phi$.
Moreover, if measures $\mu$ and $\nu$ have densities $\mu(x)$ and $\nu(x)$, then we have
\begin{equation*}
 \begin{aligned}
   (\mu*\nu)(x) = \int_{\mb{R}^{d}} \mu(x-y)\nu(y) \d y
   = \int_{\mb{R}^{d}} \mu(y)\nu(x-y) \d y.
 \end{aligned}
\end{equation*}
\end{definition}

Subsequently, we go over the concept of gauge-type functions.

\begin{definition}\label{def:gauge}
Let $\mr{dist}$ be a metric on $\mc{P}_{2}(\mb{R}^{d})$ such that
$(\mc{P}_{2}(\mb{R}^{d}),\mr{dist})$ is complete.
A map $\mc{G}:\mc{P}_{2}(\mb{R}^{d})\times\mc{P}_{2}(\mb{R}^{d})\ra[0,\infty)$
is said to be a \textbf{gauge-type function} on the complete metric space
$(\mc{P}_{2}(\mb{R}^{d}),\mr{dist})$ provided that:
\begin{enumerate}[\rm \bfseries(i)]
\item
$\mc{G}(\mu,\mu)=0$, for every $\mu\in\mc{P}_{2}(\mb{R}^{d})$;
\item
$\mc{G}$ is continuous on $\mc{P}_{2}(\mb{R}^{d})\times\mc{P}_{2}(\mb{R}^{d})$;
\item
for all $\varepsilon>0$, there exists $\eta_{\epsilon}>0$ such that,
for each $(\mu,\nu)\in\mc{P}_{2}(\mb{R}^{d})\times\mc{P}_{2}(\mb{R}^{d})$,
we have $\mc{G}(\mu,\nu)\leq\eta_{\epsilon}$ implies $\mr{dist}(\mu,\nu)\leq\varepsilon$.
\end{enumerate}
\end{definition}


In order to generate a gauge-type function on $\mc{P}_{2}(\Rd)$,
we select a special metric on $\mc{P}_{2}(\Rd)$,
which is the Gaussian regularized 2-Wasserstein distance.
For any $\mu,\nu\in\mc{P}_{2}(\mb{R}^{d})$,
the Gaussian regularized 2-Wasserstein distance is defined as
\begin{equation*}
 \begin{aligned}
   \mc{W}_{2}^{\sigma}(\mu,\nu) :=
   \mc{W}_{2}(\mu*\mc{N}_{\sigma},\nu*\mc{N}_{\sigma}),
 \end{aligned}
\end{equation*}
where $\mc{N}_{\sigma}:=\mc{N}(0,\sigma^2I_{d})\in\mc{P}_{2}(\mb{R}^{d})$
is the normal distribution with variance matrix $\sigma^2I_{d}$ for some $\sigma>0$.
Let $\varphi_{\sigma}(x)=
(2\pi\sigma^{2})^{-\frac{d}{2}}\exp\left( -\frac{|x|^{2}}{2\sigma^{2}} \right)$
denote the density function of $\mc{N}_{\sigma}$,
and then $(\mu*\mc{N}_{\sigma})(z)=\int_{\Rd} \varphi_{\sigma}(z-x)\d \mu (x)$.
For any $\sigma>0$, $(\mc{P}_{2}(\mb{R}^{d}),\mc{W}_{2}^{\sigma})$ is a complete space
and it induces the same topology as $(\mc{P}_{2}(\mb{R}^{d}),\mc{W}_{2})$
(cf. \cite[Lemma 4.2]{CGKPR2024}).
Fix $\sigma>0$ and let $B_{0}=(-1,1]^{d}$.
For $l\geq 0$, $\wp_{l}$ denotes the partition of $B_{0}$ into
$2^{dl}$ translations of $(-2^{-l},2^{-l}]^{d}$
and for every $n\geq 1$, $B_{n}=(-2^{n},2^{n}]^{d}\backslash(-2^{n-1},2^{n-1}]^{d}$.
Let $\theta_{n,l}=2^{-(4n+2dl)}$.
For convenience, we introduce two notations:
\begin{align}
  &\ a_{n,l}(\mu,\nu,B) :=
  \left| \left( \mu*\mc{N}_{\sigma}-\nu*\mc{N}_{\sigma} \right)
  \left( (2^{n}B)\cap B_{n} \right) \right|, \notag \\ 
  &\ b_{n,l}(\mu,\nu,B) :=  \sqrt{|a_{n,l}(\mu,\nu,B)|^{2}
  + 2 \theta_{n,l} \sqrt{1-\theta_{n,l}} a_{n,l}(\mu,\nu,B) + \theta_{n,l}^{2}}
  - \theta_{n,l}, \label{eq:Shorthand-2}
\end{align}
where $2^{n}B:=\{ 2^{n}x\in\Rd:x\in B \}$.
Then, we define
\begin{equation}\label{eq:gauge}
 \begin{aligned}
   \mc{G}(\mu,\nu) :=
   \sum_{n\geq 0} 2^{2n} \sum_{l\geq 0} 2^{-2l} \sum_{B\in \wp_{l}} b_{n,l}(\mu,\nu,B)
   \triangleq \sum_{n,l,B} 2^{2(n-l)} b_{n,l}(\mu,\nu,B).
 \end{aligned}
\end{equation}
The following result ensures that $\mc{G}$ is actually
a gauge-type function on $\mc{P}_{2}(\mb{R}^{d})$.

\begin{lemma} 
$\mc{G}(\mu,\nu)$ is a gauge-type function on $\mc{P}_{2}(\mb{R}^{d})$,
which is endowed with the metric $\mc{W}_{2}^{\sigma}$.
\end{lemma}

\begin{proof}
The proof of items (i) and (ii) in \Cref{def:gauge}
is identical to \cite[Lemma 4.4]{CGKPR2024}.
So we only need to explain item (iii) in \Cref{def:gauge}.

Our objective is to show the following:
for all $\epsilon>0$, there exists $\eta_{\epsilon}>0$ such that,
for each $(\mu,\nu)\in\mc{P}_{2}(\mb{R}^{d})\times\mc{P}_{2}(\mb{R}^{d})$,
the inequality $\mc{G}(\mu,\nu)\leq\eta_{\epsilon}$ implies
$\sum_{n,l,B} 2^{2(n-l)} |a_{n,l}(\mu,\nu,B)|\leq\epsilon$.
Recalling \eqref{eq:Shorthand-2}, we obtain
\begin{align*}
  a_{n,l}(\mu,\nu,B) = &\
  \sqrt{|b_{n,l}(\mu,\nu,B)|^{2} + 2 \theta_{n,l} b_{n,l}(\mu,\nu,B)
  + \theta_{n,l}^{2}(1-\theta_{n,l})} - \theta_{n,l}\sqrt{1-\theta_{n,l}} \\
  \leq &\ b_{n,l}(\mu,\nu,B)
  + \sqrt{2 \theta_{n,l} b_{n,l}(\mu,\nu,B) + \theta_{n,l}^{2}(1-\theta_{n,l})}
  - \theta_{n,l}\sqrt{1-\theta_{n,l}} \\
  \leq &\ b_{n,l}(\mu,\nu,B) + \sqrt{2 \theta_{n,l} b_{n,l}(\mu,\nu,B)},
\end{align*}
where we have used the elementary inequality
$\sqrt{x+y}\leq\sqrt{x}+\sqrt{y}$, for every $x,y\geq 0$.
Since $2^{2(n-l)}b_{n,l}(\mu,\nu,B)\leq\mc{G}(\mu,\nu)\leq\eta_{\epsilon}$
and $\theta_{n,l}=2^{-(4n+2dl)}$, we get
\begin{align*}
  \sum_{n,l,B} 2^{2(n-l)} a_{n,l}(\mu,\nu,B) \leq &\
  \sum_{n,l,B} 2^{2(n-l)} b_{n,l}(\mu,\nu,B) +
  \sqrt{2\eta_{\epsilon}} \sum_{n,l,B} 2^{2(n-l)} 2^{-3n-(d-1)l} \\
  \leq &\ \eta_{\epsilon} +
  \sqrt{2\eta_{\epsilon}} \sum_{n\geq 0} \sum_{l\geq0} 2^{2n} 2^{-2l} 2^{dl} 2^{-3n-(d-1)l}
  = \eta_{\epsilon} + 4\sqrt{2\eta_{\epsilon}}.
\end{align*}
Actually, for every $\mu,\nu\in\mc{P}_{2}(\Rd)$, the following inequality holds
\begin{align}\label{eq:Wassertein-bound}
  \left| \mc{W}_{2}(\mu,\nu) \right|^{2} \leq C_{d}
  \sum_{n,l,B} 2^{2(n-l)} \left| \mu((2^{n}B)\cap B_{n}) - \nu((2^{n}B)\cap B_{n}) \right|,
\end{align}
where $C_{d}$ is a constant depending only on $d$.
Inequality \eqref{eq:Wassertein-bound} follows from \cite[Lemma 5.12]{CD2018-1},
which gives an upper bound of the distance $\mc{W}_{2}$.
It follows from \eqref{eq:Wassertein-bound} that
$\mc{W}_{2}^{\sigma}(\mu,\nu)\leq\sqrt{C_{d}\epsilon}$
provided that $\mc{G}(\mu,\nu)\leq\eta_{\epsilon}=(\sqrt{8+\epsilon}-2\sqrt{2})^{2}$.

Thus, the proof of this lemma is finalized.
\end{proof}

The subsequent corollary provides crucial properties of the gauge-type function $\mc{G}$,
which will be applied in the next section to prove the comparison theorem.

\begin{corollary}\label{cor:gauge}
For every $\mu,\nu\in\mc{P}_{2}(\Rd)$
and fixed $\phi\in C_b(\Rd)$,
the following inequality holds:
\begin{align*}
  \left| \int_{\Rd} |x|^{2} \d \left( \mu*\varphi_{\sigma} - \nu*\varphi_{\sigma} \right)(x) \right|
  + \left| \int_{\Rd} \phi(x) \d \left( \mu*\varphi_{\sigma} - \nu*\varphi_{\sigma} \right)(x) \right|
  \leq C_{d} \mc{G}(\mu,\nu),
\end{align*}
where $C_{d}$ is a constant depending only on $d$.
\end{corollary}

\begin{proof}
Recalling that $B_{n}=(-2^{n},2^{n}]^{d}\backslash(-2^{n-1},2^{n-1}]^{d}$,
we have $\frac{|x|^{2}}{d}\leq2^{2n},\forall\,x\in B_{n}$.
Hence, we get
\begin{align*}
  &\ \mc{G}(\mu,\nu) =
  \sum_{n,l,B} 2^{2(n-l)} b_{n,l}(\mu,\nu,B)
  \geq \sum_{n,l,B} 2^{2(n-l)} a_{n,l}(\mu,\nu,B) \\
  = &\ \sum_{n\geq 0} 2^{2n} \sum_{l\geq 0} 2^{-2l}
  \left| \left(\mu*\mc{N}_{\sigma} - \nu*\mc{N}_{\sigma}\right)(B_{n}) \right|
  = \frac{4}{3} \sum_{n\geq 0} 2^{2n}
  \left| \left(\mu*\mc{N}_{\sigma} - \nu*\mc{N}_{\sigma}\right)(B_{n}) \right| \\
  \geq &\ \frac{4}{3d} \left| \int_{\Rd} |x|^{2} \d
  \left( \mu*\varphi_{\sigma} - \nu*\varphi_{\sigma} \right)(x) \right|.
\end{align*}
When $\phi\in C_b(\Rd)$ is fixed,
we also observe that $\phi(x)\leq C_{d} 2^{2n},\forall\,x\in B_{n}$,
which leads to
\begin{align*}
  \left| \int_{\Rd} \phi(x) \d \left( \mu*\varphi_{\sigma} - \nu*\varphi_{\sigma} \right)(x) \right|
  \leq C_{d} \mc{G}(\mu,\nu).
\end{align*}
\end{proof}

Next, we calculate the derivatives of the gauge-type function $\mc{G}$.

\begin{proposition}\label{prop:gauge-derive}
Let $\nu\in\mc{P}_{2}(\mb{R}^{d})$ and $\sigma>0$ be fixed.
It holds that: \\
$\mr{(a)}$ the map $\mc{P}_{2}(\mb{R}^{d})\ni\mu\mapsto\mc{G}(\mu,\nu)$
has first-order variational derivative
{\small
\begin{align*}
   \left(\frac{\delta}{\delta\mu}\mc{G}(\mu,\nu)\right)(x) =&\
   \sum_{n,l,B} 2^{2(n-l)}
   \frac{\left(\mu*\mc{N}_{\sigma} - \nu*\mc{N}_{\sigma}\right) ((2^{n}B)\cap B_{n})
   + \theta_{n,l} \sqrt{1-\theta_{n,l}}}
   {b_{n,l}(\mu,\nu,B) + \theta_{n,l}}
   \psi_{n}^{B}(x) \\
   =&\ \sum_{n,l,B} 2^{2(n-l)}
   \frac{\int_{\mb{R}^{d}} \psi_{n}^{B}(z) \d (\mu - \nu)(z)
   + \theta_{n,l} \sqrt{1-\theta_{n,l}}}
   {b_{n,l}(\mu,\nu,B) + \theta_{n,l}}
   \psi_{n}^{B}(x)\in L^{\infty}(\mu;\mb{R}^{1}), 
\end{align*}}
where
$\psi_{n}^{B}(x) := \int_{(2^{n}B)\cap B_{n}}\varphi_{\sigma}(z-x) \d z,~\forall\,x\in\mb{R}^{d}$; \\
$\mr{(b)}$ the map $\mc{P}_{2}(\mb{R}^{d})\ni\mu\mapsto\mc{G}(\mu,\nu)$
has second-order variational derivative
\begin{align*}
   \left(\frac{\delta^{2}}{\delta\mu^{2}}\mc{G}(\mu,\nu)\right)(x,y) =
   \sum_{n,l,B} 2^{2(n-l)} \frac{\theta_{n,l}^{3}}
   {|b_{n,l}(\mu,\nu,B) + \theta_{n,l}|^{3}}
   \psi_{n}^{B}(x)\psi_{n}^{B}(y)
   \in L^{\infty}(\mu\otimes\mu;\mb{R}^{1}); 
\end{align*}
$\mr{(c)}$ the map $\mc{P}_{2}(\mb{R}^{d})\ni\mu\mapsto\mc{G}(\mu,\nu)$
has first-order Lions derivative
{\small
\begin{align*}
   \left(\partial_{\mu}\mc{G}(\mu,\nu)\right)(x) =
   \sum_{n,l,B} 2^{2(n-l)}
   \frac{\int_{\mb{R}^{d}} \psi_{n}^{B}(z) \d (\mu - \nu)(z)
   + \theta_{n,l} \sqrt{1-\theta_{n,l}}}
   {b_{n,l}(\mu,\nu,B) + \theta_{n,l}}
   \partial_{x}\psi_{n}^{B}(x)\in L^{\infty}(\mu;\mb{R}^{d}); 
\end{align*}}
$\mr{(d)}$ for each $\mu\in\mc{P}_{2}(\mb{R}^{d})$,
the map $\mb{R}^{d}\ni x\mapsto\left(\partial_{\mu}\mc{G}(\mu,\nu)\right)(x)$
is differentiable w.r.t. $x$ with
{\footnotesize
\begin{align*}
   \left(\partial_{x}\partial_{\mu}\mc{G}(\mu,\nu)\right)(x) =
   \sum_{n,l,B} 2^{2(n-l)}
   \frac{\int_{\mb{R}^{d}} \psi_{n}^{B}(z) \d (\mu - \nu)(z)
   + \theta_{n,l} \sqrt{1-\theta_{n,l}}}
   {b_{n,l}(\mu,\nu,B) + \theta_{n,l}}
   \partial_{xx}^{2}\psi_{n}^{B}(x)\in L^{\infty}(\mu;\mb{R}^{d\times d}); 
\end{align*}}
$\mr{(e)}$ the map $\mc{P}_{2}(\mb{R}^{d})\ni\mu\mapsto\mc{G}(\mu,\nu)$
has second-order Lions derivative
{\footnotesize
\begin{align*}
   \left(\partial_{\mu\mu}^{2}\mc{G}(\mu,\nu)\right)(x,y) =
   \sum_{n,l,B} 2^{2(n-l)} \frac{\theta_{n,l}^{3}}
   {|b_{n,l}(\mu,\nu,B) + \theta_{n,l}|^{3}}
   (\partial_{x}\psi_{n}^{B}(x))
   \left( \partial_{y}^{\top}\psi_{n}^{B}(y) \right)
   \in L^{\infty}(\mu\otimes\mu;\mb{R}^{d\times d}); 
\end{align*}}
$\mr{(f)}$ there exists a constant $C_{d}$,
which is solely dependent on the dimension $d$ and can vary with each occurrence,
such that
\begin{align*}
   &\ \left| \frac{\delta}{\delta\mu}\mc{G}(\mu,\nu) \right|,
   \left| \partial_{\mu}\mc{G}(\mu,\nu) \right|,
   \left| \partial_{x}\partial_{\mu}\mc{G}(\mu,\nu) \right|
   \leq C_{d}(1 + |x|^{2}),\\
   &\ \left| \frac{\delta^{2}}{\delta\mu^{2}}\mc{G}(\mu,\nu) \right|,
   \left| \partial_{\mu\mu}^{2}\mc{G}(\mu,\nu) \right|
   \leq C_{d}\left(1 + |x|^{2}\right)\left(1 + |y|^{2}\right).
\end{align*}
\end{proposition}

\begin{proof}
Let $\Psi\in C_{b}^{2}(\mb{R}^{n})$ and $\phi^{1},...,\phi^{n}\in C_{b}^{2}(\mb{R}^{d})$.
For a cylindrical function $\upsilon(\mu) = \Psi(\mu(\phi^{1}),...,\mu(\phi^{n}))$,
it follows that

\begin{align*}
  \frac{\delta\upsilon(\mu)}{\delta\mu}(x) =&\ \sum_{i=1}^{n}\partial_{i}
  \Psi(\mu(\phi^{1}),...,\mu(\phi^{n})) \phi^{i}(x), \\
  \frac{\delta^{2}\upsilon(\mu)}{\delta\mu^{2}}(x,y) =&\ \sum_{i,j=1}^{n}\partial_{ij}^{2}
  \Psi(\mu(\phi^{1}),...,\mu(\phi^{n})) \phi^{i}(x)\phi^{j}(y), \\
  \partial_{\mu}\upsilon(\mu)(x) =&\ \sum_{i=1}^{n}\partial_{i}
  \Psi(\mu(\phi^{1}),...,\mu(\phi^{n})) \partial_{x}\phi^{i}(x), \\
  \partial_{\mu\mu}^{2}\upsilon(\mu)(x,y) =&\ \sum_{i,j=1}^{n}\partial_{ij}^{2}
  \Psi(\mu(\phi^{1}),...,\mu(\phi^{n})) \partial_{x}\phi^{i}(x)\partial_{y}^{\top}\phi^{j}(y).
\end{align*}
Therefore, we directly calculate the derivatives in items (a)--(e) from above.
Similar calculations can also be found in \cite[Lemma 4.4]{CGKPR2024}.

The estimates of each order derivative in item (f)
are derived from the structure of $B_{n}$.
Recalling that $B_{n}=(-2^{n},2^{n}]^{d}\backslash(-2^{n-1},2^{n-1}]^{d}$,
we have $2^{2n-2}\leq\frac{|x|^{2}}{d},\forall\,x\in B_{n}$.
Here, we just examine the second-order variational derivative as an illustration.
Notice that
\begin{align*}
  &\ \mu\otimes\mu \left(
  \sum_{n,l,B} 2^{2(n-l)}\frac{\theta_{n,l}^{3}}
  {|b_{n,l}(\mu,\nu,B) + \theta_{n,l}|^{3}}
  \psi_{n}^{B}(x)\psi_{n}^{B}(y) \right) \\
  \leq &\ \frac{4}{3} \sum_{n\geq0} 2^{2n}
  \left[ \int_{\Rd} \int_{B_{n}} \varphi_{\sigma}(z-x) \d z \d \mu(x) \right]
  \cdot \left[ \int_{\Rd} \int_{B_{n}} \varphi_{\sigma}(z-y) \d z \d \mu(y) \right] \\
  \leq &\ \frac{16}{3d}
  \left[ \int_{\Rd}\int_{\Rd\backslash(-2^{N-1},2^{N-1}]^{d}} |z|^{2}
  \varphi_{\sigma}(z-x) \d z \d \mu(x) \right] \\
  &\ \cdot \left[ \int_{\Rd}\int_{\Rd\backslash(-2^{N-1},2^{N-1}]^{d}} |z|^{2}
  \varphi_{\sigma}(z-y) \d z \d \mu(y) \right] \\
  \leq &\ \frac{16}{3d}
  \left[ \iint_{\mb{R}^{2d}} \mbf{1}_{\{|\kappa+x|\geq 2^{N-1}\}} |\kappa+x|^{2}
  \varphi_{\sigma}(\kappa) \d \kappa \d \mu(x) \right] \\
  &\ \cdot \left[ \iint_{\mb{R}^{2d}} \mbf{1}_{\{|\tau+y|\geq 2^{N-1}\}} |\tau+y|^{2}
  \varphi_{\sigma}(\tau) \d \tau \d \mu(y) \right] \\
  \leq &\ \frac{64}{3d}
  \left[ \iint_{\mb{R}^{2d}} \left(\mbf{1}_{\{|\kappa|\geq 2^{N-3/2}\}} |\kappa|^{2}
  + \mbf{1}_{\{|x|\geq 2^{N-3/2}\}} |x|^{2} \right)
  \varphi_{\sigma}(\kappa) \d \kappa \d \mu(x) \right] \\
  &\ \cdot \left[ \iint_{\mb{R}^{2d}} \left(\mbf{1}_{\{|\tau|\geq 2^{N-3/2}\}} |\tau|^{2}
  + \mbf{1}_{\{|y|\geq 2^{N-3/2}\}} |y|^{2} \right)
  \varphi_{\sigma}(\tau) \d \tau \d \mu(y) \right] \\
  \leq &\ C_{d} \left(1+\int_{\Rd} |x|^{2} \d \mu(x)\right)
  \left(1+\int_{\Rd} |y|^{2} \d \mu(y)\right),
\end{align*}
where the penultimate inequality follows from the elementary inequality
\begin{align*}
  |x+y|^{2}\mbf{1}_{\{|x+y|\geq 2^{N-1}\}} \leq
  |x|^{2}\mbf{1}_{\{|x+y|\geq 2^{N-3/2}\}} + |y|^{2}\mbf{1}_{\{|x+y|\geq 2^{N-3/2}\}},
  \quad\forall\,x,y\in\Rd.
\end{align*}
Then, we conclude that
$\frac{\delta^{2}}{\delta\mu^{2}}\mc{G}(\mu,\nu)\in L^{\infty}(\mu\otimes\mu;\mb{R}^{1})$.
Thus far, we have proven the primary conclusion of this proposition.
We omit the remaining proof because we can demonstrate other estimates in a similar manner.
\end{proof}

\begin{remark}\label{rmk:gauge}
In fact, both this paper and \cite{CGKPR2024}
use the upper bound of $\mc{W}_{2}^{}$ (i.e., inequality \eqref{eq:Wassertein-bound})
to construct a gauge-type function.
The gauge-type function in \cite{CGKPR2024} has the form
\begin{align*}
  \tilde{\mc{G}}(\mu,\nu)
  = \sum_{n,l,B} 2^{2(n-l)} \left(
  \sqrt{|a_{n,l}(\mu,\nu,B)|^{2}+\theta_{n,l}^{2}} - \theta_{n,l} \right).
\end{align*}
It is vital to emphasize that the gauge-type function stated above
is \emph{not applicable} to the situation discussed in this paper.
The reason for this is that
$\frac{\delta^{2}}{\delta\mu^{2}}\tilde{\mc{G}}(\mu,\nu)
\notin L^{\infty}(\mu\otimes\mu;\mb{R}^{1})$, i.e.,
\begin{align*}
   \mu\otimes\mu \left(
   \frac{\delta^{2}}{\delta\mu^{2}}\tilde{\mc{G}}(\mu,\nu) \right)=
   \sum_{n,l,B} 2^{2(n-l)}\frac{\theta_{n,l}^{2}}
   {\left( |a_{n,l}(\mu,\nu,B)|^{2} + \theta_{n,l}^{2} \right)^{\frac{3}{2}}} \left|\mu(\psi_{n}^{B})\right|^{2}
\end{align*}
is unbounded as $\theta_{n,l}$ approaches 0.
\end{remark}

We also need to define a Gaussian regularized entropy functional
\begin{equation*}
 \begin{aligned}
   \mc{E}(\mu*\mc{N}_{\sigma})
   := \int_{\mb{R}^{d}} \log\left( \mu*\mc{N}_{\sigma} \right) \d (\mu*\mc{N}_{\sigma})
   = \int_{\mb{R}^{d}}
   (\mu*\varphi_{\sigma})(x) \log\left( (\mu*\varphi_{\sigma})(x) \right) \d x.
 \end{aligned}
\end{equation*}
The following proposition illustrates the derivatives of the entropy functional.

\begin{proposition}\label{prop:entropy-derive}
Let $\sigma>0$ be fixed. 
If we force $\mc{I}(\mu*\mc{N}_{\sigma})<\infty$,
it holds that: \\
$\mr{(a)}$ the map $\mc{P}_{2}(\mb{R}^{d})\ni\mu\mapsto\mc{E}(\mu*\mc{N}_{\sigma})$
has first-order variational derivative
\begin{equation*}
 \begin{aligned}
   \left(\frac{\delta}{\delta\mu}\mc{E}(\mu*\mc{N}_{\sigma})\right)(x) =
   1 + \int_{\mb{R}^{d}} \varphi_{\sigma}(z-x)
   \log\left( (\mu*\varphi_{\sigma})(z) \right) \d z \in
   L^{\infty}(\mu;\mb{R}^{1});
 \end{aligned}
\end{equation*}
$\mr{(b)}$ the map $\mc{P}_{2}(\mb{R}^{d})\ni\mu\mapsto\mc{E}(\mu*\mc{N}_{\sigma})$
has second-order variational derivative
\begin{equation*}
 \begin{aligned}
   \left(\frac{\delta^{2}}{\delta\mu^{2}}\mc{E}(\mu*\mc{N}_{\sigma})\right)(x,y) =
   1 + \int_{\mb{R}^{d}} \frac{\varphi_{\sigma}(z-x)\varphi_{\sigma}(z-y)}
   {(\mu*\varphi_{\sigma})(z)} \d z \in
   L^{\infty}(\mu\otimes\mu;\mb{R}^{1});
 \end{aligned}
\end{equation*}
$\mr{(c)}$ the map $\mc{P}_{2}(\mb{R}^{d})\ni\mu\mapsto\mc{E}(\mu*\mc{N}_{\sigma})$
has first-order Lions derivative
\begin{equation*}
 \begin{aligned}
   \left(\partial_{\mu}\mc{E}(\mu*\mc{N}_{\sigma})\right)(x) =
   \int_{\mb{R}^{d}} \partial_{x}\varphi_{\sigma}(z-x)
   \log\left( (\mu*\varphi_{\sigma})(z) \right) \d z \in
   L^{\infty}(\mu;\mb{R}^{d});
 \end{aligned}
\end{equation*}
$\mr{(d)}$ for each $\mu\in\mc{P}_{2}(\mb{R}^{d})$,
the map $\mb{R}^{d}\ni x\mapsto\left(\partial_{\mu}\mc{E}(\mu*\mc{N}_{\sigma})\right)(x)$
is differentiable w.r.t. $x$ with
\begin{equation*}
 \begin{aligned}
   \left(\partial_{x}\partial_{\mu}\mc{E}(\mu*\mc{N}_{\sigma})\right)(x) =
   \int_{\mb{R}^{d}} \partial_{xx}^{2}\varphi_{\sigma}(z-x)
   \log\left( (\mu*\varphi_{\sigma})(z) \right) \d z \in
   L^{\infty}(\mu;\mb{R}^{d\times d});
 \end{aligned}
\end{equation*}
$\mr{(e)}$ the map $\mc{P}_{2}(\mb{R}^{d})\ni\mu\mapsto\mc{E}(\mu*\mc{N}_{\sigma})$
has second-order Lions derivative
\begin{equation*}
 \begin{aligned}
   \left(\partial_{\mu\mu}^{2}\mc{E}(\mu*\mc{N}_{\sigma})\right)(x,y) =
   \int_{\mb{R}^{d}} \frac{\partial_{x}\varphi_{\sigma}(z-x)\partial_{y}^{\top}\varphi_{\sigma}(z-y)}
   {(\mu*\varphi_{\sigma})(z)} \d z \in
   L^{\infty}(\mu\otimes\mu;\mb{R}^{d\times d});
 \end{aligned}
\end{equation*}
$\mr{(f)}$ the map $\mc{P}_{2}(\mb{R}^{d})\ni\mu\mapsto\tilde{\mc{E}}(\mu*\mc{N}_{\sigma})$
has the following first-order variational and Lions derivatives
\begin{align*}
   &\ \frac{\delta}{\delta\mu}\tilde{\mc{E}}(\mu*\mc{N}_{\sigma}) =
   \frac{\delta}{\delta\mu}\mc{E}(\mu*\mc{N}_{\sigma})
   + \pi\int_{\Rd}|z|^{2}\varphi_{\sigma}(z-x) \d z, \\
   &\ \partial_{\mu}\tilde{\mc{E}}(\mu*\mc{N}_{\sigma}) =
   \partial_{\mu}\mc{E}(\mu*\mc{N}_{\sigma})
   + \pi\int_{\Rd}|z|^{2}\partial_{x}\varphi_{\sigma}(z-x) \d z, \\
   &\ \partial_{x}\partial_{\mu}\tilde{\mc{E}}(\mu*\mc{N}_{\sigma}) =
   \partial_{x}\partial_{\mu}\mc{E}(\mu*\mc{N}_{\sigma})
   + \pi\int_{\Rd}|z|^{2}\partial_{xx}^{2}\varphi_{\sigma}(z-x) \d z,
\end{align*}
while the second-order variational and Lions derivatives of $\tilde{\mc{E}}(\mu*\mc{N}_{\sigma})$
are the same as $\mc{E}(\mu*\mc{N}_{\sigma})$.
\end{proposition}

\begin{proof}
We will only provide the calculation procedure of variational derivatives
because the Lions derivatives of functional $\mc{E}(\mu*\mc{N}_{\sigma})$
is relatively easy to determine.
The treatment of the convolution term in $\mc{E}(\mu*\mc{N}_{\sigma})$ is critical to the calculation.

Denote $\mc{E}(\mu*\mc{N}_{\sigma})$ as
\begin{equation*}
 \begin{aligned}
   \digamma^{1}[\mu(\cdot)] \triangleq \mc{E}(\mu*\mc{N}_{\sigma}) =
   \int_{\mb{R}^{d}}
   \left( \int_{\mb{R}^{d}} \varphi_{\sigma}(x-z)\mu(z) \d z \right)
   \log \left( \int_{\mb{R}^{d}} \varphi_{\sigma}(x-z)\mu(z) \d z \right) \d x.
 \end{aligned}
\end{equation*}
For a test function $\phi\in\mc{D}(\mb{R}^{d})$
that satisfies $\int_{\mb{R}^{d}} \phi(x)\d x =1$, we get
{\small
\begin{align*}
   &\ \frac{\delta}{\delta\mu} \digamma^{1}[\mu] =
   \left. \frac{\d}{\d\epsilon}\digamma^{1}[\mu(x)+\epsilon\phi(x)] \right|_{\epsilon=0} \\
   =&\ \left. \frac{\d}{\d\epsilon} \left\{ \int_{\mb{R}^{d}}
   \left( \int_{\mb{R}^{d}} \varphi_{\sigma}(x-z)[\mu(z)+\epsilon\phi(z)] \d z \right)
   \log \left( \int_{\mb{R}^{d}} \varphi_{\sigma}(x-z)[\mu(z)+\epsilon\phi(z)] \d z \right) \d x
   \right\} \right|_{\epsilon=0} \\
   =&\ \left.\int_{\mb{R}^{d}} \left(\int_{\mb{R}^{d}}\varphi_{\sigma}(x-z)\mu(z)\d z\right)
   \frac{\int_{\mb{R}^{d}} \varphi_{\sigma}(x-z)\phi(z) \d z}
   {\int_{\mb{R}^{d}} \varphi_{\sigma}(x-z)[\mu(z)+\epsilon\phi(z)] \d z}
   \d x \right|_{\epsilon=0} \\
   &\ + \left.\int_{\mb{R}^{d}} \left(\int_{\mb{R}^{d}}\varphi_{\sigma}(x-z)\phi(z)\d z\right)
   \log \left(\int_{\mb{R}^{d}} \varphi_{\sigma}(x-z)[\mu(z)+\epsilon\phi(z)] \d z\right)
   \d x \right|_{\epsilon=0} \\
   &\ + \left.\int_{\mb{R}^{d}} \left(\int_{\mb{R}^{d}}\varphi_{\sigma}(x-z)\epsilon\phi(z)\d z\right)
   \frac{\int_{\mb{R}^{d}} \varphi_{\sigma}(x-z)\phi(z) \d z}
   {\int_{\mb{R}^{d}} \varphi_{\sigma}(x-z)[\mu(z)+\epsilon\phi(z)] \d z}
   \d x \right|_{\epsilon=0} \\
   =&\ \int_{\mb{R}^{d}} \left(\int_{\mb{R}^{d}}\varphi_{\sigma}(x-z)\phi(z)\d z\right)
   [1+\log\left( (\mu*\varphi_{\sigma})(x) \right)] \d x \\
   =&\ \int_{\mb{R}^{d}} \left(\int_{\mb{R}^{d}} \varphi_{\sigma}(x-z)
   [1+\log\left( (\mu*\varphi_{\sigma})(x) \right)] \d x \right) \phi(z)\d z.
\end{align*}}Then, we have
{\small
\begin{align*}
   &\ \left(\frac{\delta}{\delta\mu}\mc{E}(\mu*\mc{N}_{\sigma})\right)(x)
   = \int_{\mb{R}^{d}} \varphi_{\sigma}(z-x)[1+\log\left( (\mu*\varphi_{\sigma})(z) \right)] \d z \\
   =&\ 1 + \int_{\mb{R}^{d}} \varphi_{\sigma}(z-x)\log\left( (\mu*\varphi_{\sigma})(z) \right) \d z
   \equiv 1 + \int_{\mb{R}^{d}} \varphi_{\sigma}(z-x)
   \log\left( \int_{\mb{R}^{d}}\varphi_{\sigma}(z-y)\mu(y)\d y \right) \d z.
\end{align*}}

Next, we calculate the second-order variational derivative of $\mc{E}(\mu*\mc{N}_{\sigma})$.
Denote the first-order variational derivative of $\mc{E}(\mu*\mc{N}_{\sigma})$ as
\begin{align*}
   \digamma^{2}[\mu(\cdot)] \triangleq &\
   \left(\frac{\delta}{\delta\mu}\mc{E}(\mu*\mc{N}_{\sigma})\right)(x) =
   1 + \int_{\mb{R}^{d}} \varphi_{\sigma}(z-x)\log\left( (\mu*\varphi_{\sigma})(z) \right) \d z \\
   \equiv&\ \int_{\mb{R}^{d}}\mu(z)\d z + \int_{\mb{R}^{d}} \varphi_{\sigma}(z-x)
   \log\left( \int_{\mb{R}^{d}}\varphi_{\sigma}(z-y)\mu(y)\d y \right) \d z.
\end{align*}
By direct calculation, we obtain
{\small
\begin{align*}
   &\ \frac{\delta}{\delta\mu} \digamma^{2}[\mu] =
   \left. \frac{\d}{\d\epsilon}\digamma^{2}[\mu(y)+\epsilon\phi(y)] \right|_{\epsilon=0} \\
   =&\ \left. \frac{\d}{\d\epsilon} \left\{
   \int_{\mb{R}^{d}}[\mu(z)+\epsilon\phi(z)]\d z + \int_{\mb{R}^{d}} \varphi_{\sigma}(z-x)
   \log\left( \int_{\mb{R}^{d}}\varphi_{\sigma}(z-y)[\mu(y)+\epsilon\phi(y)]\d y \right) \d z \right\}
   \right|_{\epsilon=0} \\
   =&\ \int_{\mb{R}^{d}} \phi(z) \d z
   + \left. \int_{\mb{R}^{d}} \varphi_{\sigma}(z-x)
   \frac{\int_{\mb{R}^{d}} \varphi_{\sigma}(z-y)\phi(y) \d y}
   {\int_{\mb{R}^{d}} \varphi_{\sigma}(z-y)[\mu(y)+\epsilon\phi(y)] \d y}
   \d z \right|_{\epsilon=0} \\
   =&\ \int_{\mb{R}^{d}} \phi(z) \d z + \int_{\mb{R}^{d}} \varphi_{\sigma}(z-x)
   \frac{(\phi*\varphi_{\sigma})(z)}{(\mu*\varphi_{\sigma})(z)} \d z
   = \int_{\mb{R}^{d}} \phi(y) \left[1 + \int_{\mb{R}^{d}}
   \frac{\varphi_{\sigma}(z-x)\varphi_{\sigma}(z-y)}{(\mu*\varphi_{\sigma})(z)} \d z \right] \d y.
\end{align*}}Hence, we find
\begin{equation*}
 \begin{aligned}
   \left(\frac{\delta}{\delta\mu} \digamma^{2}[\mu]\right)(y)
   = \left(\frac{\delta^{2}}{\delta\mu^{2}}\mc{E}(\mu*\mc{N}_{\sigma})\right)(x,y)
   = 1 + \int_{\mb{R}^{d}}
   \frac{\varphi_{\sigma}(z-x)\varphi_{\sigma}(z-y)}{(\mu*\varphi_{\sigma})(z)} \d z.
 \end{aligned}
\end{equation*}

Ultimately,
we shall argue that each order derivatives
are bounded in the appropriate spaces.
We observe that
\begin{align*}
   &\ \partial_{x}\varphi_{\sigma}(z-x)=\frac{1}{\sigma^{2}}
   \int_{\Rd} (z-x)\varphi_{\sigma}(z-x) \d z,\\
   &\ \partial_{xx}^{2}\varphi_{\sigma}(z-x)=\frac{1}{\sigma^{2}} I_{d}
   \int_{\Rd} \varphi_{\sigma}(z-x) \d z
   - \frac{1}{\sigma^{4}}\int_{\Rd} (z-x)\otimes(z-x) \varphi_{\sigma}(z-x) \d z,
\end{align*}
where $(z-x)\otimes(z-x)$ is the $d\times{d}$ matrix with $ij$-component equal to
$(z_{i}-x_{i})(z_{j}-x_{j})$.
Straightforward computations show that
\begin{align*}
  \mu\left(\frac{\delta}{\delta\mu}\mc{E}(\mu*\mc{N}_{\sigma})\right) =&\
  1 + \iint_{\mb{R}^{2d}} \varphi_{\sigma}(z-x)
  \log\left( (\mu*\varphi_{\sigma})(z) \right) \d z \d \mu(x) \\
  = &\ 1 + \int_{\Rd} \log\left( (\mu*\varphi_{\sigma})(z) \right) \d (\mu*\mc{N}_{\sigma})(z)
  = 1+\mc{E}(\mu*\varphi_{\sigma}), \\
  \mu\otimes\mu\left(\frac{\delta^{2}}{\delta\mu^{2}}\mc{E}(\mu*\mc{N}_{\sigma})\right) =&\
  1 + \iiint_{\mb{R}^{3d}} \frac{\varphi_{\sigma}(z-x)\varphi_{\sigma}(z-y)}
  {(\mu*\varphi_{\sigma})(z)} \d z \d \mu(x) \d \mu(y) = 2, \\
  \left| \mu \left(\partial_{\mu}\mc{E}(\mu*\mc{N}_{\sigma})\right) \right| =&\
  \left| \iint_{\mb{R}^{2d}} \partial_{x}\varphi_{\sigma}(z-x)
  \log\left( (\mu*\varphi_{\sigma})(z) \right) \d z \d \mu(x) \right| \\
  \leq &\ \int_{\Rd} \left| \partial_{z} (\mu*\varphi_{\sigma})(z) \right| \d z, \\
  \left| \mu \left(\partial_{x}\partial_{\mu}\mc{E}(\mu*\mc{N}_{\sigma})\right)
  \right| =&\
  \left| \iint_{\mb{R}^{2d}} \partial_{xx}^{2}\varphi_{\sigma}(z-x)
  \log\left( (\mu*\varphi_{\sigma})(z) \right) \d z \d \mu(x) \right| \\
  \leq &\ | \mc{I}(\mu*\varphi_{\sigma}) |, \\
\end{align*}
and
\begin{align*}
  &\ \left| \mu\otimes\mu\left(\partial_{\mu\mu}^{2}\mc{E}(\mu*\mc{N}_{\sigma})\right)
  \right|= \iiint_{\mb{R}^{3d}} \left|
  \frac{\partial_{x}\varphi_{\sigma}(z-x)\partial_{y}^{\top}\varphi_{\sigma}(z-y)}
  {(\mu*\varphi_{\sigma})(z)} \d z \d \mu(x) \d \mu(y) \right| \\
  \leq &\ \left( \int_{\Rd} \int_{\Rd} |z-x|^{2} \varphi_{\sigma}(z-x) \d \mu(x)
  \int_{\Rd} |z-y|^{2} \varphi_{\sigma}(z-y) \d \mu(y) \d z \right)^{\frac{1}{2}},
\end{align*}
where we have used the Cauchy-Schwartz and Jensen inequalities.
According to \Cref{rmk:entropy-Fisher},
it is evident that the above results are well-defined,
provided that $\mc{I}(\mu*\varphi_{\sigma})<\infty$.
Therefore, the proof ends here.
\end{proof}

\begin{remark}\label{rmk:Gauss}
We can conclude that a Gaussian regularization procedure is necessary
based on the proofs of \Cref{prop:gauge-derive,prop:entropy-derive}.
On the one hand, we know that if $\nu\in\mc{P}_{2}(\Rd)$ is fixed,
the square of the 2-Wasserstein distance function
$\mc{P}_{2}(\Rd)\ni\mu\mapsto|\mc{W}_{2}(\mu,\nu)|^{2}$
may not be Lions differentiable
(see A Sobering Counter-Example in Section 5.5.1 of \cite{CD2018-1}).
Therefore, we require a smooth gauge-type function, i.e.,
an upper bound of the Gaussian regularized 2-Wasserstein distance $\mc{W}_{2}^{\sigma}$.
On the other hand, it is problematic for us to directly compute
two variational derivatives of the entropy functional
$\mc{P}_{2}(\Rd)\ni\mu\mapsto\mc{E}(\mu)$.
In fact, it holds that
\begin{align*}
  \left( \frac{\delta}{\delta\mu}\mc{E}(\mu) \right)(x) = 1 + \log\mu(x)
  ~~\text{and}~~
  \left( \frac{\delta^{2}}{\delta\mu^{2}}\mc{E}(\mu) \right)(x,y) =
  1 + \frac{\delta(x-y)}{\mu(x)}.
\end{align*}
Here, the second-order variational derivatives are computed through
the Dirac delta function $\delta(x-y)\in\mc{D}^{\prime}(\Rd)$
in place of a generic test function $\phi\in\mc{D}(\Rd)$.
This methodology is quite common within physics.
Needless to say, the derivatives of the entropy functional
will be exceedingly intricate in the absence of Gaussian regularization.
\end{remark}

Our main result of this section is the following theorem.

\begin{theorem}[Maximum Principle]\label{thm:max}
Assume that $\upsilon^{1},\upsilon^{2}:\mc{P}_{2}(\mb{R}^{d})\ra\mb{R}^{d}$
are bounded $\mc{W}_{1}$-Lipschitz continuous functions, and fix $\alpha,\beta>0$.
Then the bivariate functional
$\Phi_{\alpha,\beta}(\mu,\nu):\mc{P}_{2}(\mb{R}^{d})\times\mc{P}_{2}(\mb{R}^{d})\ra\mb{R}^{1}$
defined by
\begin{equation*}
 \begin{aligned}
   \Phi_{\alpha,\beta}(\mu,\nu) := \upsilon^{1}(\mu) - \upsilon^{2}(\nu)
   - \frac{1}{2\alpha}\mc{G}(\mu,\nu)
   - \beta \left\{ \tilde{\mc{E}}(\mu*\mc{N}_{\sigma}) + \tilde{\mc{E}}(\nu*\mc{N}_{\sigma}) \right\}
 \end{aligned}
\end{equation*}
achieves its supremum.
Any maximum point $(\bar{\mu},\ubar{\mu})$ belongs to
$\mc{P}_{\mr{ac},2}(\mb{R}^{d})\times\mc{P}_{\mr{ac},2}(\mb{R}^{d})$
and satisfying the following properties:
$\mr{(i)}$
$\bar{\mu},\ubar{\mu}\in L_{\mr{loc}}^{1}(\Rd)$;
$\mr{(ii)}$ $\mc{I}(\mc{\bar{\mu}}*\mc{N}_{\sigma})$ and $\mc{I}(\mc{\ubar{\mu}}*\mc{N}_{\sigma})$
are finite.
\end{theorem}

\begin{proof}
The proof of this theorem is similar to \cite[Proposition 3.1]{DS2024}.
The primary distinction is that the optimal transport theory
is not required to demonstrate the well-posedness of Fisher information.
This is due to the fact that we do not employ the fluid mechanics interpretation of $\mc{W}_{2}$
(cf. \cite[Lecture 17]{ABS2021}).

Let us provide several key details of the proof.
We can show that the functional $\Phi_{\alpha,\beta}(\cdot,\cdot)$ is bounded from above
by leveraging the fact that the gauge-type function $\mc{G}$ is non-negative
and that $\upsilon^{1}$ and $\upsilon^{2}$ are also bounded.
Therefore, the definition of the supremum leads us to the following conclusion:
for $n=1,2,...$, there exists a sequence
$(\mu_{n},\nu_{n})\in\mc{P}_{2}(\Rd)\times\mc{P}_{2}(\Rd)$,
such that
\begin{align*}
 \sup_{(\mu,\nu)\in\mc{P}_{2}(\Rd)\times\mc{P}_{2}(\Rd)}\Phi_{\alpha,\beta}(\mu,\nu)
 - \frac{1}{n}\leq \Phi_{\alpha,\beta}(\mu_{n},\nu_{n}).
\end{align*}
It follows from the inequality
$\Phi_{\alpha,\beta}\left(\mc{N}_{\frac{1}{\sqrt{\pi}}},\mc{N}_{\frac{1}{\sqrt{\pi}}}\right)
\leq \Phi_{\alpha,\beta}(\mu_{n},\nu_{n}) + \frac{1}{n}$ that
\begin{align*}
 &\ \frac{1}{2\alpha} \mc{G}(\mu_{n},\nu_{n})
 + \frac{\beta\pi}{2} \int_{\Rd} |x|^{2} \d (\mu_{n}*\varphi_{\sigma})(x)
 + \frac{\beta\pi}{2} \int_{\Rd} |x|^{2} \d (\nu_{n}*\varphi_{\sigma})(x) \\
 \leq &\ \upsilon^{1}(\delta_{0}) - \upsilon^{1}(\mc{N}_{\frac{1}{\sqrt{\pi}}})
 + \upsilon^{2}(\mc{N}_{\frac{1}{\sqrt{\pi}}}) - \upsilon^{2}(\delta_{0})
 + 2 \tilde{\mc{E}}(\mc{N}_{\frac{1}{\sqrt{\pi}}}*\mc{N}_{\sigma}) + \beta d \log 2 + \frac{1}{n} \\
 &\ + \mr{Lip}(\upsilon^{1};\mc{W}_{1})\int_{\Rd} |x| \d (\mu_{n}*\varphi_{\sigma})(x)
 + \mr{Lip}(\upsilon^{2};\mc{W}_{1})\int_{\Rd} |x| \d ( \nu_{n}*\varphi_{\sigma} )(x)
 \leq C_{d}.
\end{align*}
In light of the uniformly bounded square moments of the sequences
$\{\mu_{n}*\mc{N}_{\sigma}\}_{n=1}^{\infty}$ and $\{\nu_{n}*\mc{N}_{\sigma}\}_{n=1}^{\infty}$,
it is simple to infer that these two sequences are \emph{tight}.
This indicates that the two sets of probability measures do not escape to infinity.
Furthermore, according to Prokhorov's theorem,
there are two subsequences $\{\mu_{n_k}*\mc{N}_{\sigma}\}_{k=1}^{\infty}$
and $\{\nu_{n_k}*\mc{N}_{\sigma}\}_{k=1}^{\infty}$
and probability measures $\bar{\mu}\in\mc{P}(\Rd)$ and $\ubar{\mu}\in\mc{P}(\Rd)$,
such that $\mu_{n_k}*\mc{N}_{\sigma}\rightharpoonup\bar{\mu}$
and $\nu_{n_k}*\mc{N}_{\sigma}\rightharpoonup\ubar{\mu}$.
This is to say that \emph{Prokhorov's theorem} expressed the tightness as compactness
(and hence weak convergence).

The famous \emph{Skorokhod's representation theorem} establishes that
a weakly convergent sequence of probability measures,
whose limit measure is well-behaved, can be represented as
a pointwise convergent sequence of random variables on a common probability space.
Then there are two sets of $\Rd$-valued random variables,
$\{\Xi_{k}^{1}\}_{k=1}^{\infty}$ and $\{\Xi_{k}^{2}\}_{k=1}^{\infty}$,
that are defined on the same probability space $(\Omega,\mc{F},\mb{P})$.
For all $n=1,2,...$, there are probability laws of $\Xi_{k}^{1}$ and $\Xi_{k}^{2}$
that are $\mu_{n_k}$ and $\nu_{n_k}$, respectively,
and $\Xi_{n}^{1}\ra\Xi^{1}$ and $\Xi_{n}^{2}\ra\Xi^{2},\mb{P}$-a.s.
As a result, Skorohod's representation theorem together with \emph{Fatou's lemma} imply that
\begin{align*}
 \int_{\Rd} |x|^{2}\d \bar{\mu}(x) = \mb{E} \left| \Xi_{}^{1} \right|^{2}
 \leq \liminf_{k\ra\infty} \mb{E} \left| \Xi_{k}^{1} \right|^{2}
 \leq \liminf_{k\ra\infty} \int_{\Rd} |x|^{2} \d \mu_{n_k}(x) < \infty.
\end{align*}
Hence, we get $\bar{\mu},\ubar{\mu}\in\mc{P}_{2}(\Rd)$.
Simultaneously, we notice the following inequality:
\begin{align*}
 \int_{\{ |x|\geq \zeta \}} |x|^{p} \d \mu_{n_k}(x)
 \leq \int_{\{ |x|\geq \zeta \}} \frac{|x|^{2-p}}{\zeta^{2-p}}|x|^{p} \d \mu_{n_k}(x)
 \leq \frac{1}{\zeta^{2-p}} \int_{\Rd} |x|^{2} \d \mu_{n_k}(x).
\end{align*}
When we fix $p\in[1,2)$ and set $\zeta$ to infinity,
it is evident that the right-hand side of the above inequality approaches zero.
This suggests that $\mu_{n_k}$ has uniformly integrable $p$-moments for $p\in[1,2)$.
Ultimately, the \emph{dominated convergence theorem} illustrates that
$\lim_{k\ra\infty} \mb{E}\left| \Xi_{k}^{1} \right| = \mb{E}\left| \Xi_{}^{1} \right|$.
In conclusion, we can use $\mu_{n_k}\rightharpoonup\bar{\mu}$
as well as $\int_{\Rd} |x| \d \mu_{n_k}(x)\ra\int_{\Rd} |x| \d \mu(x)$ to deduce that
$\mc{W}_{1}(\mu_{n_k},\mu)\ra0$ as $k\ra\infty$.
All the mathematical tools used above are available in \cite{Bogachev2007}.

The \emph{weak lower semi-continuity} of the entropy functional
(cf. \cite[Corollary 15.7]{ABS2021}) ensures that we obtain
\begin{align*}
 \sup_{(\mu,\nu)\in\mc{P}_{2}(\Rd)\times\mc{P}_{2}(\Rd)}\Phi_{\alpha,\beta}(\mu,\nu)
 \leq \limsup_{k\ra\infty}\Phi_{\alpha,\beta}(\mu_{n_k},\nu_{n_k})
 \leq \Phi_{\alpha,\beta}(\bar{\mu},\ubar{\mu}).
\end{align*}
Consequently, $(\bar{\mu},\ubar{\mu})$ is indeed the maximum point of
functional $\Phi_{\alpha,\beta}(\cdot,\cdot)$.

In particular, the method of \cite[Lemma 3.2 \& 3.3]{DS2024}
can be applied to elucidate the fact that $\bar{\mu}$ and $\ubar{\mu}$ satisfy item (i).
The proof of item (ii) is also straightforward,
as the maximum point $(\bar{\mu},\ubar{\mu})$ has already been sufficiently smoothed.
Based on the definition of Fisher information, we obtain
\begin{align*}
 \mc{I}(\bar{\mu}*\mc{N}_{\sigma})
 =&\ \int_{\Rd} \frac{|\partial_{x}(\bar{\mu}*\varphi_{\sigma})(x)|^{2}}
 {(\bar{\mu}*\varphi_{\sigma})(x)} \d x
 = \frac{1}{\sigma^{4}} \int_{\Rd} \frac{ \left|\int_{\Rd}(z-x)\varphi_{\sigma}(x-z)
 \d \bar{\mu}(z) \right|^{2}}{(\bar{\mu}*\varphi_{\sigma})(x)} \d x \\
 \leq &\ \frac{1}{\sigma^{4}} \int_{\Rd}
 \frac{\int_{\Rd} |z-x|^{2} \varphi_{\sigma}(x-z)\d \bar{\mu}(z)
 \int_{\Rd} \varphi_{\sigma}(x-z)\d \bar{\mu}(z)}{(\bar{\mu}*\varphi_{\sigma})(x)} \d x \\
 =&\ \frac{1}{\sigma^{4}} \iint_{\mb{R}^{2d}} |z-x|^{2}\varphi_{\sigma}(x-z)
 \d \bar{\mu}(z)\d x \leq C_{d},
\end{align*}
where the first inequality is a result of the Cauchy-Schwartz inequality
and the last inequality holds due to $\bar{\mu}\in\mc{P}_{2}(\Rd)$.
This concludes the proof of the theorem.
\end{proof}

\begin{remark}
The \emph{Crandall-Ishii lemma} (a type of maximum principle)
is frequently employed to prove the comparison theorem in finite-dimensional
stochastic optimal control problems with \emph{unbounded} and Lipschitz continuous diffusion terms
(see, e.g., \cite[Theorem E.10]{FGS2017} and a holistic review thereof).
The monograph \cite{FGS2017} also provides a generalization of
the Crandall-Ishii lemma in Hilbert space,
and the notion of \emph{$B$-continuous viscosity solutions} is needed.
However, the Crandall-Ishii lemma is presently unable to be applied to
the general scenario in which the state space is the Wasserstein space.
Thus, \cite{BEZ2023,CTQ2025,DJS2025}
are restricted to the study of state-independent correlated noise.
Moreover, the bivariate functional $\Phi_{\alpha,\beta}$,
which includes both a gauge-type function and an entropy functional as penalty terms,
cannot directly facilitate the generalization of the Crandall-Ishii lemma
due to the lack of finite-dimensional natures in these terms.
\end{remark}

\section{Comparison theorem and uniqueness}
\label{sec:comparison}

\begin{theorem}[Comparison]\label{thm:comparison}
Suppose that \Cref{hyp:BC,hyp:Lip1} hold.
Consider bounded and $\mc{W}_{1}$-Lipschitz continuous functions
$u^{1}:\mc{P}_{2}(\mb{R}^{d})\ra\mb{R}^{1}$
and $u^{2}:\mc{P}_{2}(\mb{R}^{d})\ra\mb{R}^{1}$ with $u^{1}$ (resp. $u^{2}$)
being a viscosity subsolution (resp. supersolution) to \Cref{eq:HJB-1}.
Then, it holds that $u^{1}\leq u^{2}$ on $\mc{P}_{2}(\mb{R}^{d})$.
\end{theorem}

\begin{proof}
We employ proof by contradiction to establish this theorem.
Suppose that $u^{1} - u^{2}>0$ on $\mc{P}_{2}(\mb{R}^{d})$.
Since $\tilde{\mc{E}}\geq0$, then we find $\beta_{0}>0$ such that
\begin{equation}\label{eq:contradiction}
 \begin{aligned}
   \varrho := \sup_{\mu\in\mc{P}_{2}(\mb{R}^{d})} \left( u^{1}(\mu) - u^{2}(\mu)
   - 2\beta_{0} \tilde{\mc{E}}(\mu*\mc{N}_{\sigma}) \right) >0.
 \end{aligned}
\end{equation}
Thus there exists $\hat{\mu}\in\mc{P}_{2}(\Rd)$ such that
$\varrho=u^{1}(\hat{\mu})-u^{2}(\hat{\mu})-2\beta_{0}\tilde{\mc{E}}(\hat{\mu}*\mc{N}_{\sigma})$.
We split the rest of the proof into several steps.

\textbf{Step 1.} (Doubling variables with entropy penalization).
Let $\alpha,\beta>0$ be fixed.
From \Cref{thm:max} we know that there exists
$(\bar{\mu},\ubar{\mu})\in\mc{P}_{\mr{ac},2}(\mb{R}^{d})\times\mc{P}_{\mr{ac},2}(\mb{R}^{d})$
satisfying
\begin{align*}
 0 < \varrho = \Phi_{\alpha,\beta}(\hat{\mu},\hat{\mu}) \leq
 \sup_{(\mu,\nu)\in\mc{P}_{2}(\Rd)\times\mc{P}_{2}(\Rd)}\Phi_{\alpha,\beta}(\mu,\nu)
 = \Phi_{\alpha,\beta}(\bar{\mu},\ubar{\mu}).
\end{align*}
Setting $M_{\alpha,\beta}
:=\sup_{(\mu,\nu)\in\mc{P}_{2}(\Rd)\times\mc{P}_{2}(\Rd)}\Phi_{\alpha,\beta}(\mu,\nu)$,
results in the map $\beta\mapsto M_{\alpha,\beta}$ being nonincreasing.
It follows from the inequality
\begin{align*}
 M_{\alpha,\beta}=\Phi_{\alpha,\beta}(\bar{\mu},\ubar{\mu})
 \leq M_{\alpha,\frac{\beta}{2}}-\frac{\beta}{2}\left\{
 \tilde{\mc{E}}(\bar{\mu}*\mc{N}_{\sigma}) + \tilde{\mc{E}}(\ubar{\mu}*\mc{N}_{\sigma}) \right\}
\end{align*}
and $\lim_{\beta\ra0}\left[ M_{\alpha,\frac{\beta}{2}} - M_{\alpha,\beta} \right]=0$ that
\begin{align*}
 \lim_{\beta\ra0} \beta\left\{
 \tilde{\mc{E}}(\bar{\mu}*\mc{N}_{\sigma}) + \tilde{\mc{E}}(\ubar{\mu}*\mc{N}_{\sigma})
 \right\}=0.
\end{align*}
Recalling \eqref{eq:entropy-bound} we find
\begin{align*}
 \lim_{\beta\ra0} \beta\left\{
 \int_{\Rd} |x|^{2} \d (\bar{\mu}*\varphi_{\sigma})(x)
 + \int_{\Rd} |x|^{2} \d (\ubar{\mu}*\varphi_{\sigma})(x)
 \right\}=0.
\end{align*}
Since $\Phi_{\alpha,\beta}(\bar{\mu},\bar{\mu}) + \Phi_{\alpha,\beta}(\ubar{\mu},\ubar{\mu})
\leq 2\Phi_{\alpha,\beta}(\bar{\mu},\ubar{\mu})$, we have
\begin{align*}
 \frac{1}{2\alpha} \mc{G}(\bar{\mu},\ubar{\mu}) \leq &\
 u^{1}(\bar{\mu}) - u^{1}(\ubar{\mu}) + u^{2}(\bar{\mu}) - u^{2}(\ubar{\mu})
 \leq \left( \mr{Lip}(u^{1};\mc{W}_{1}) + \mr{Lip}(u^{2};\mc{W}_{1}) \right)
 \mc{W}_{1}(\bar{\mu},\ubar{\mu}) \\
 \leq &\ \left( \mr{Lip}(u^{1};\mc{W}_{1}) + \mr{Lip}(u^{2};\mc{W}_{1}) \right)
 \mc{W}_{2}(\bar{\mu},\ubar{\mu}).
\end{align*}
It is immediately apparent that
\begin{align*}
 \mc{W}_{2}^{2}(\bar{\mu},\ubar{\mu}) \leq
 C_{d} \mc{G}(\bar{\mu},\ubar{\mu}) \leq
 2\alpha C_{d} \left( \mr{Lip}(u^{1};\mc{W}_{1}) + \mr{Lip}(u^{2};\mc{W}_{1})
 \right)\mc{W}_{2}^{}(\bar{\mu},\ubar{\mu}).
\end{align*}
Hence we obtain $\mc{W}_{2}^{}(\bar{\mu},\ubar{\mu})\leq
2\alpha C_{d} \left( \mr{Lip}(u^{1};\mc{W}_{1}) + \mr{Lip}(u^{2};\mc{W}_{1}) \right)$.
Then, we find
\begin{align*}
 \lim_{\alpha\ra0} \frac{1}{2\alpha} \mc{G}(\bar{\mu},\ubar{\mu}) =0.
\end{align*}

\textbf{Step 2.} (Viscosity property).
We note that
\begin{equation*}
 \begin{aligned}
   \Phi_{\alpha,\beta}(\mu,\ubar{\mu}) = u^{1}(\mu) - \phi^{1}(\mu)
   ~~ \text{and} ~~
   \Phi_{\alpha,\beta}(\bar{\mu},\nu) = -u^{2}(\nu) + \phi^{2}(\nu),
 \end{aligned}
\end{equation*}
where
\begin{align*}
   \phi^{1}(\mu) =&\ u^{2}(\ubar{\mu}) + \frac{1}{2\alpha} \mc{G}(\mu,\ubar{\mu})
   + \left\{ \tilde{\mc{E}}(\mu*\mc{N}_{\sigma}) + \tilde{\mc{E}}(\ubar{\mu}*\mc{N}_{\sigma}) \right\}
   \in \mc{C}_{L}^{2}(\mc{P}_{2}(\mb{R}^{d})), \\
   \phi^{2}(\nu) =&\ u^{1}(\bar{\mu}) - \frac{1}{2\alpha} \mc{G}(\bar{\mu},\nu)
   - \left\{ \tilde{\mc{E}}(\bar{\mu}*\mc{N}_{\sigma}) + \tilde{\mc{E}}(\nu*\mc{N}_{\sigma}) \right\}
   \in \mc{C}_{L}^{2}(\mc{P}_{2}(\mb{R}^{d})).
\end{align*}
It is easily seen that
{\small
\begin{align*}
  &\ \partial_{\mu}\phi^{1}(\mu) = \frac{1}{2\alpha}\partial_{\mu}\mc{G}(\mu,\ubar{\mu})
  + \beta\partial_{\mu}\mc{E}(\mu*\mc{N}_{\sigma})
  + \beta\pi\int_{\Rd}|z|^{2}\partial_{x}\varphi_{\sigma}(z-x) \d z, \\
  &\ \partial_{\mu}\phi^{2}(\nu) = \frac{1}{2\alpha}\partial_{\mu}\mc{G}(\bar{\mu},\nu)
  - \beta\partial_{\mu}\mc{E}(\nu*\mc{N}_{\sigma})
  - \beta\pi\int_{\Rd}|z|^{2}\partial_{x}\varphi_{\sigma}(z-x) \d z; \\
  &\ \partial_{x}\partial_{\mu}\phi^{1}(\mu) =
  \frac{1}{2\alpha}\partial_{x}\partial_{\mu}\mc{G}(\mu,\ubar{\mu})
  + \beta\partial_{x}\partial_{\mu}\mc{E}(\mu*\mc{N}_{\sigma})
  + \beta\pi\int_{\Rd}|z|^{2}\partial_{xx}^{2}\varphi_{\sigma}(z-x) \d z, \\
  &\ \partial_{x}\partial_{\mu}\phi^{2}(\nu) =
  \frac{1}{2\alpha}\partial_{x}\partial_{\mu}\mc{G}(\bar{\mu},\nu)
  - \beta\partial_{x}\partial_{\mu}\mc{E}(\nu*\mc{N}_{\sigma})
  - \beta\pi\int_{\Rd}|z|^{2}\partial_{xx}^{2}\varphi_{\sigma}(z-x) \d z; \\
  &\ \frac{\delta^{2}\phi^{1}(\mu)}{\delta\mu^{2}} =
  \frac{1}{2\alpha}\frac{\delta^{2}}{\delta\mu^{2}}\mc{G}(\mu,\ubar{\mu})
  + \beta\frac{\delta^{2}}{\delta\mu^{2}}\mc{E}(\mu*\mc{N}_{\sigma}),
  \frac{\delta^{2}\phi^{2}(\nu)}{\delta\mu^{2}} =
  - \frac{1}{2\alpha}\frac{\delta^{2}}{\delta\mu^{2}}\mc{G}(\bar{\mu},\nu)
  - \beta\frac{\delta^{2}}{\delta\mu^{2}}\mc{E}(\nu*\mc{N}_{\sigma}); \\
  &\ \partial_{x}\frac{\delta^{2}\phi^{1}(\mu)}{\delta\mu^{2}} =
  \frac{1}{2\alpha}\partial_{x}\frac{\delta^{2}}{\delta\mu^{2}}\mc{G}(\mu,\ubar{\mu})
  + \beta\partial_{x}\frac{\delta^{2}}{\delta\mu^{2}}\mc{E}(\mu*\mc{N}_{\sigma}), \\
  &\ \partial_{x}\frac{\delta^{2}\phi^{2}(\nu)}{\delta\mu^{2}} =
  - \frac{1}{2\alpha}\partial_{x}\frac{\delta^{2}}{\delta\mu^{2}}\mc{G}(\bar{\mu},\nu)
  - \beta\partial_{x}\frac{\delta^{2}}{\delta\mu^{2}}\mc{E}(\nu*\mc{N}_{\sigma}); \\
  &\ \partial_{\mu\mu}^{2}\phi^{1}(\mu) =
  \frac{1}{2\alpha}\partial_{\mu\mu}^{2}\mc{G}(\mu,\ubar{\mu})
  + \beta\partial_{\mu\mu}^{2}\mc{E}(\mu*\mc{N}_{\sigma}), ~
  \partial_{\mu\mu}^{2}\phi^{2}(\nu) =
  - \frac{1}{2\alpha}\partial_{\mu\mu}^{2}\mc{G}(\bar{\mu},\nu)
  - \beta\partial_{\mu\mu}^{2}\mc{E}(\nu*\mc{N}_{\sigma}).
\end{align*}}Since $u^{1} - \phi^{1}$ attains its maximum at $\bar{\mu}$
and $u^{2} - \phi^{2}$ attains its minimum at $\ubar{\mu}$,
the viscosity property gives
\begin{align*}
   F \left( \bar{\mu},u^{1}(\bar{\mu}),\partial_{\mu}\phi^{1}(\bar{\mu}),
   \partial_{x}\partial_{\mu}\phi^{1}(\bar{\mu}),
   \frac{\delta^{2}\phi^{1}(\bar{\mu})}{\delta\mu^{2}},
   \partial_{x}\frac{\delta^{2}\phi^{1}(\bar{\mu})}{\delta\mu^{2}},
   \partial_{\mu\mu}^{2}\phi^{1}(\bar{\mu}) \right)
   \leq 0,
\end{align*}
and
\begin{align*}
   F \left( \ubar{\mu},u^{2}(\ubar{\mu}),\partial_{\mu}^{2}\phi(\ubar{\mu}),
   \partial_{x}\partial_{\mu}\phi^{2}(\ubar{\mu}),
   \frac{\delta^{2}\phi^{2}(\ubar{\mu})}{\delta\mu^{2}},
   \partial_{x}\frac{\delta^{2}\phi^{2}(\ubar{\mu})}{\delta\mu^{2}},
   \partial_{\mu\mu}^{2}\phi^{2}(\ubar{\mu}) \right)
   \geq 0.
 \end{align*}
Putting these inequalities together, we get
{\small
\begin{align*}
   &\ u^{1}(\bar{\mu}) - u^{2}(\ubar{\mu}) \leq
   \sup_{\gamma\in\Gamma} \left| L(\bar{\mu},\gamma) - L(\ubar{\mu},\gamma) \right| \\
   &\ + \sup_{\gamma\in\Gamma} \left|
   \bar{\mu}\left( \bra b(x,\gamma),\partial_{\mu}\phi^{1}(\bar{\mu}) \ket_{\mb{R}^{d}}
   + \frac{1}{2}\mr{Trace}\left[ a^{}(x,\gamma)\partial_{x}\partial_{\mu}\phi^{1}(\bar{\mu})
   \right] \right) \right.\\
   &\ \qquad \quad \left. - \ubar{\mu}\left( \bra b(y,\gamma),
   \partial_{\mu}\phi^{2}(\ubar{\mu}) \ket_{\mb{R}^{d}}
   + \frac{1}{2}\mr{Trace}\left[ a^{}(y,\gamma)\partial_{x}\partial_{\mu}\phi^{1}(\ubar{\mu})
   \right] \right) \right| \\
   &\ + \frac{1}{2} \bar{\mu}\otimes\bar{\mu}\left[
   \bra h(x)-\bar{\mu}(h),h(y)-\bar{\mu}(h) \ket_{\mb{R}^{q}}
   \frac{\delta^{2}\phi^{1}(\bar{\mu})}{\delta\mu^{2}} \right] \\
   &\ \qquad \quad - \frac{1}{2} \ubar{\mu}\otimes\ubar{\mu}\left[
   \bra h(x)-\ubar{\mu}(h),h(y)-\ubar{\mu}(h) \ket_{\mb{R}^{q}}
   \frac{\delta^{2}\phi^{2}(\ubar{\mu})}{\delta\mu^{2}} \right]  \\
   &\ + \bar{\mu}\otimes\bar{\mu}\left(
   \bra h(x)-\bar{\mu}(h),\sigma^{2,\top}(y)
   \partial_{x}\frac{\delta^{2}\phi^{1}(\bar{\mu})}{\delta\mu^{2}} \ket_{\mb{R}^{q}} \right) \\
   &\ \qquad \quad - \ubar{\mu}\otimes\ubar{\mu}\left(
   \bra h(x)-\ubar{\mu}(h),\sigma^{2,\top}(y)
   \partial_{x}\frac{\delta^{2}\phi^{2}(\ubar{\mu})}{\delta\mu^{2}} \ket_{\mb{R}^{q}} \right) \\
   &\ + \frac{1}{2} \bar{\mu}\otimes\bar{\mu}\left( \mr{Trace} \left[
   \sigma^{2}(x)\sigma^{2,\top}(y) \partial_{\mu\mu}^{2}\phi^{1}(\bar{\mu}) \right] \right)
   - \frac{1}{2} \ubar{\mu}\otimes\ubar{\mu}\left( \mr{Trace} \left[
   \sigma^{2}(x)\sigma^{2,\top}(y) \partial_{\mu\mu}^{2}\phi^{2}(\ubar{\mu}) \right] \right) \\
   \triangleq &\ I^{1} +I^{2} + I^{3} + I^{4} + I^{5}.
\end{align*}}

\textbf{Step 3.} (Estimating $I^{1}$).
Since $L$ is Lipschitz continuous  with respect to $\mc{W}_{1}$, we obtain
\begin{equation*}
 \begin{aligned}
   I^{1} \leq \mr{Lip}(L;\mc{W}_{1}) \mc{W}_{1}(\bar{\mu},\ubar{\mu})
   \leq \mc{W}_{2}(\bar{\mu},\ubar{\mu}).
 \end{aligned}
\end{equation*}

\textbf{Step 4.} (Estimating $I^{2}$).
We enlarge $I^{2}$ into four parts, namely
{\small
\begin{align*}
   I^{2} \leq &\ \frac{1}{2\alpha} \sup_{\gamma\in\Gamma} \left| \int_{\mb{R}^{d}} \left[
   \bra b(x,\gamma),\partial_{\mu}\mc{G}(\bar{\mu},\ubar{\mu}) \ket_{\mb{R}^{d}}
   + \frac{1}{2}\mr{Trace}\left[ a^{}(x,\gamma)\partial_{x}\partial_{\mu}\mc{G}(\bar{\mu},\ubar{\mu})
   \right] \right] \cdot \d \left( \bar{\mu} - \ubar{\mu}\right) (x) \right| \\
   &\ + \beta \sup_{\gamma\in\Gamma} \int_{\mb{R}^{d}} \left| \bra b(x,\gamma),
   \int_{\mb{R}^{d}} \partial_{x}\varphi_{\sigma}(z-x)
   \log\left( (\bar{\mu}*\varphi_{\sigma})(z) \right) \d z \ket \right. \\
   &\ \qquad \qquad \quad ~ \left. + \frac{1}{2}\mr{Trace} \bra a^{}(x,\gamma),
   \int_{\mb{R}^{d}} \partial_{xx}^{2}\varphi_{\sigma}(z-x)
   \log\left( (\bar{\mu}*\varphi_{\sigma})(z) \right) \d z \ket \right| \d \bar{\mu}(x) \\
   &\ + \beta \sup_{\gamma\in\Gamma} \int_{\mb{R}^{d}} \left| \bra b(x,\gamma),
   \int_{\mb{R}^{d}} \partial_{x}\varphi_{\sigma}(z-x)
   \log\left( (\ubar{\mu}*\varphi_{\sigma})(z) \right) \d z \ket \right. \\
   &\ \qquad \qquad \quad ~ \left. + \frac{1}{2}\mr{Trace} \bra a^{}(x,\gamma),
   \int_{\mb{R}^{d}}\partial_{xx}^{2}\varphi_{\sigma}(z-x)
   \log\left( (\ubar{\mu}*\varphi_{\sigma})(z) \right) \d z \ket \right| \d \ubar{\mu}(x) \\
   &\ + 2 \beta \sup_{\gamma\in\Gamma} \int_{\mb{R}^{d}} \left| \bra b(x,\gamma),
   \int_{\Rd} |z|^{2}\partial_{x}\varphi_{\sigma}(z-x)\d z \ket_{\Rd} \right. \\
   &\ \qquad \qquad \quad ~ \left. + \frac{1}{2}\mr{Trace} \left[ a^{}(x,\gamma)
   \int_{\mb{R}^{d}} |z|^{2} \partial_{xx}^{2}\varphi_{\sigma}(z-x) \d z \right] \right|
   \d (\bar{\mu} + \ubar{\mu}) (x) \\
   \triangleq&\ I^{2,a} + I^{2,b} + I^{2,c} + I^{2,d}.
\end{align*}}Obviously, the integral in $I^{2,d}$ is finite,
so we have $I^{2,d}\lesssim_{d}\beta$,
where we use $x\lesssim_{d} y$ as shorthand for the inequality $x\leq C_{d}y$,
with the constant $C_{d}$ depending only on $d$
(note that the constant may differ at each use of the $\lesssim_{d}$ notation).
Because $I^{2,b}$ and $I^{2,c}$ have similar forms,
it is sufficient to estimate $I^{2,b}$.
By utilizing integration by parts and applying the Cauchy-Schwarz and Jensen inequalities, we obtain
{\small
\begin{align*}
   \frac{I^{2,b}}{\beta} \leq &\ \sup_{\gamma\in\Gamma} \left| \int_{\mb{R}^{d}}
   \partial_{z}\log((\bar{\mu}*\varphi_{\sigma})(z))
   \int_{\Rd} b(x,\gamma)\varphi_{\sigma}(z-x) \d \bar{\mu}(x) \d z \right| \\
   &\ + \sup_{\gamma\in\Gamma} \left| \int_{\mb{R}^{d}}
   \partial_{z}\log((\bar{\mu}*\varphi_{\sigma})(z))
   \int_{\Rd} a^{}(x,\gamma)\partial_{x}\varphi_{\sigma}(z-x) \d \bar{\mu}(x) \d z \right| \\
   =&\ \sup_{\gamma\in\Gamma} \left| \int_{\mb{R}^{d}} \frac{\d z}{(\bar{\mu}*\varphi_{\sigma})(z)}
   \int_{\Rd} \frac{(y-z)}{\sigma^{2}}\varphi_{\sigma}(z-y) \d \bar{\mu}(y)
   \int_{\Rd} b(x,\gamma)\varphi_{\sigma}(z-x) \d \bar{\mu}(x) \right| \\
   &\ + \sup_{\gamma\in\Gamma} \left| \int_{\mb{R}^{d}}
   \frac{\d z}{(\bar{\mu}*\varphi_{\sigma})(z)}
   \int_{\Rd} \frac{(y-z)}{\sigma^{2}}\varphi_{\sigma}(z-y) \d \bar{\mu}(y)
   \int_{\Rd} a^{}(x,\gamma)\partial_{x}\varphi_{\sigma}(z-x) \d \bar{\mu}(x) \d z \right| \\
   \leq &\ \frac{1}{\sigma^{2}} \sup_{\gamma\in\Gamma} \left| \int_{\Rd}
   \int_{\Rd} |y-z|^{2}\varphi_{\sigma}(z-y)\d y
   \int_{\Rd} |b(x,\gamma)|^{2}\varphi_{\sigma}(z-x)\d x \d z \right|^{\frac{1}{2}} \\
   &\ + \frac{1}{\sigma^{2}} \sup_{\gamma\in\Gamma} \left| \int_{\Rd}
   \int_{\Rd} |y-z|^{2}\varphi_{\sigma}(z-y)\d y
   \int_{\Rd} |a^{}(x,\gamma)(z-x)|^{2}\varphi_{\sigma}(z-x)\d x \d z \right|^{\frac{1}{2}}
   \leq C_{d}.
\end{align*}}Hence, $I^{2,b}+I^{2,c}\lesssim_{d}\beta$.
Next, we deal with $I^{2,a}$.
Note that
\begin{align*}
   2\alpha I^{2,a} \leq &\ \sup_{\gamma\in\Gamma} \int_{\mb{R}^{d}} \left|
   \bra b(x,\gamma),\partial_{\mu}\mc{G}(\bar{\mu},\ubar{\mu}) \ket_{\mb{R}^{d}}
   \right| \cdot \d \left( \bar{\mu} - \ubar{\mu}\right) (x) \\
   &\ + \frac{1}{2} \sup_{\gamma\in\Gamma} \int_{\mb{R}^{d}}
   \left| a^{}(x,\gamma)\partial_{x}\partial_{\mu}\mc{G}(\bar{\mu},\ubar{\mu})
   \right| \cdot \d \left( \bar{\mu} - \ubar{\mu}\right) (x) \\
   \leq &\ C_{d} \sup_{\gamma\in\Gamma} \int_{\mb{R}^{d}}
   (1+|x|^{2}) \cdot \d \left( \bar{\mu} - \ubar{\mu}\right) (x).
\end{align*}
Then, it holds that
\begin{equation*}
 \begin{aligned}
   \lim_{\alpha\ra0}\lim_{\beta\ra0} I^{2,a} = 0.
 \end{aligned}
\end{equation*}

\textbf{Step 5.} (Estimating $I^{3}$).
Plugging the derivatives of $\phi^{1}$ and $\phi^{2}$ into $I^{3}$, we obtain
{\small
\begin{align*}
   I^{3} \leq &\ \frac{1}{4\alpha} \iint_{\mb{R}^{2d}}
   \bra h(x)-\bar{\mu}(h),h(y)-\bar{\mu}(h) \ket_{\mb{R}^{q}}
   \frac{\delta^{2}}{\delta\mu^{2}}\mc{G}(\bar{\mu},\ubar{\mu}) \d \bar{\mu}(x) \d \bar{\mu}(y) \\
   &\ \quad \quad \quad ~ - \frac{1}{4\alpha} \iint_{\mb{R}^{2d}}
   \bra h(x)-\ubar{\mu}(h),h(y)-\ubar{\mu}(h) \ket_{\mb{R}^{q}}
   \frac{\delta^{2}}{\delta\mu^{2}}\mc{G}(\ubar{\mu},\ubar{\mu}) \d \ubar{\mu}(x) \d \ubar{\mu}(y) \\
   &\ + \frac{\beta}{2} \bar{\mu}\otimes\bar{\mu}\left[
   \bra h(x)-\bar{\mu}(h),h(y)-\bar{\mu}(h) \ket_{\mb{R}^{q}} \right]
   + \frac{\beta}{2} \ubar{\mu}\otimes\ubar{\mu} \left[
   \bra h(x)-\ubar{\mu}(h),h(y)-\ubar{\mu}(h) \ket_{\mb{R}^{q}} \right]  \\
   &\ + \frac{\beta}{2} \iint_{\mb{R}^{2d}}
   \bra h(x)-\bar{\mu}(h),h(y)-\bar{\mu}(h) \ket_{\mb{R}^{q}}
   \int_{\mb{R}^{d}} \frac{\varphi_{\sigma}(z-x)\varphi_{\sigma}(z-y)}
   {(\bar{\mu}*\varphi_{\sigma})(z)} \d z \d \bar{\mu}(x) \d \bar{\mu}(y) \\
   &\ \quad \quad \quad ~ + \frac{\beta}{2} \iint_{\mb{R}^{2d}}
   \bra h(x)-\ubar{\mu}(h),h(y)-\ubar{\mu}(h) \ket_{\mb{R}^{q}}
   \int_{\mb{R}^{d}} \frac{\varphi_{\sigma}(z-x)\varphi_{\sigma}(z-y)}
   {(\ubar{\mu}*\varphi_{\sigma})(z)} \d z \d \ubar{\mu}(x) \d \ubar{\mu}(y) \\
   \triangleq&\ I^{3,a} + I^{3,b} + I^{3,c}.
\end{align*}}Apparently, $I^{3,b}=0$.
We first estimate $I^{3,c}$, and later $I^{3,a}$.
Using the Cauchy-Schwartz inequality, we find
\begin{align*}
 &\ \iint_{\mb{R}^{2d}}
 \bra h(x)-\bar{\mu}(h),h(y)-\bar{\mu}(h) \ket_{\mb{R}^{q}}
 \int_{\mb{R}^{d}} \frac{\varphi_{\sigma}(z-x)\varphi_{\sigma}(z-y)}
 {(\bar{\mu}*\varphi_{\sigma})(z)} \d z \d \bar{\mu}(x) \d \bar{\mu}(y) \\
 =&\ \int_{\Rd} \frac{1}{(\bar{\mu}*\varphi_{\sigma})(z)}
 \left| \int_{\Rd} (h(x)-\bar{\mu}(h)) \varphi_{\sigma}(z-x) \d \bar{\mu}(x) \right|^{2} \d z \\
 \leq &\ \int_{\Rd} \frac{1}{(\bar{\mu}*\varphi_{\sigma})(z)}
 \int_{\Rd} \left|h(x)-\bar{\mu}(h)\right|^{2} \varphi_{\sigma}(z-x) \d \bar{\mu}(x)
 \int_{\Rd} \varphi_{\sigma}(z-x) \d \bar{\mu}(x)  \d z \\
 \leq &\ C_{d} \var_{\bar{\mu}}[h],
\end{align*}
where $\var_{\bar{\mu}}[h]=\bar{\mu}(h^{2})-|\bar{\mu}(h)|^{2}$ is the variance of $h(\cdot)$.
Therefore, $I^{3,c}\lesssim_{d,q} \beta$.
According to the definition of the second-order variational derivative of $\mc{G}$,
the integrand in $I^{3,a}$ is actually $\psi_{n}^{B}(\cdot)$.
A straightforward computation shows that
\begin{align*}
 &\ \left| \iint_{\mb{R}^{2d}}
 \bra h(x)-\bar{\mu}(h),h(y)-\bar{\mu}(h) \ket_{\mb{R}^{q}}
 \psi_{n}^{B}(x)\psi_{n}^{B}(y) \d \bar{\mu}(x) \d \bar{\mu}(y) \right. \\
 &\ \left. - \iint_{\mb{R}^{2d}}
 \bra h(x)-\ubar{\mu}(h),h(y)-\ubar{\mu}(h) \ket_{\mb{R}^{q}}
 \psi_{n}^{B}(x)\psi_{n}^{B}(y) \d \ubar{\mu}(x) \d \ubar{\mu}(y) \right| \\
 =&\ \left( \int_{\Rd} (h(x)-\bar{\mu}(h)) \psi_{n}^{B}(x)
 \d \left( \bar{\mu} + \ubar{\mu} \right)(x) \right)
 \cdot \left( \int_{\Rd} (h(x)-\bar{\mu}(h)) \psi_{n}^{B}(x)
 \d \left( \bar{\mu} - \ubar{\mu} \right)(x) \right) \\
 \leq &\ C_{d} \left(
 \int_{\Rd} h(x)\psi_{n}^{B}(x) \d \left( \bar{\mu} - \ubar{\mu} \right)(x)
 + |\ubar{\mu}(h)| \int_{\Rd} \psi_{n}^{B}(y) \d \left( \bar{\mu} - \ubar{\mu} \right)(y) \right. \\
 &\ \left. \qquad+ \left| \bar{\mu}(h) - \ubar{\mu}(h) \right|
 \int_{\Rd} \psi_{n}^{B}(y) \d \bar{\mu} (y)
 \right).
\end{align*}
Since $h(x), \psi_{n}^{B}(x), h(x)\psi_{n}^{B}(x)\in C_{b}(\Rd)$,
we get $\alpha I^{3,a}\lesssim_{d}\mc{G}(\mu,\nu)$.
The results in \textbf{Step 1} and \Cref{cor:gauge} clearly indicate that
\begin{equation*}
 \begin{aligned}
   \lim_{\alpha\ra0}\lim_{\beta\ra0} I^{3,a} = 0.
 \end{aligned}
\end{equation*}

\textbf{Step 6.} (Estimating $I^{4}$).
Due to the similar treatment of $I^{3}$ and $I^{4}$,
we only present two inequalities derived from the estimate of $I^{4}$.
The following inequality results from estimating
$\partial_{x}\frac{\delta^{2}}{\delta\mu^{2}}\mc{G}(\mu,\nu)$:
\begin{align*}
 &\ \left| \iint_{\mb{R}^{2d}}
 (h(x)-\bar{\mu}(h))^{\top} \sigma^{2}(y)
 \partial_{x}\psi_{n}^{B}(x)\psi_{n}^{B}(y) \d \bar{\mu}(x) \d \bar{\mu}(y) \right. \\
 &\ \left. - \iint_{\mb{R}^{2d}}
 (h(x)-\ubar{\mu}(h))^{\top} \sigma^{2}(y)
 \partial_{x}\psi_{n}^{B}(x)\psi_{n}^{B}(y) \d \ubar{\mu}(x) \d \ubar{\mu}(y) \right| \\
 \leq &\ C_{d} \left(
 \int_{\Rd} \left| h(x)-\bar{\mu}(h) \right|
 \partial_{x}\psi_{n}^{B}(x) \d \left( \bar{\mu} - \ubar{\mu} \right)(x)
 + \int_{\Rd} \left| \sigma^{2}(y) \right| \d \left( \bar{\mu} - \ubar{\mu} \right)(y) \right. \\
 &\ \left. \qquad+ \left| \bar{\mu}(h) - \ubar{\mu}(h) \right|
 \int_{\Rd} \partial_{x}\psi_{n}^{B}(x) \d \ubar{\mu} (x)
 \int_{\Rd} \left| \sigma^{2}(y) \right| \psi_{n}^{B}(y) \d \bar{\mu} (y)
 \right) \lesssim_{d}\mc{G}(\mu,\nu).
\end{align*}
Another inequality comes from the estimate of
$\partial_{x}\frac{\delta^{2}}{\delta\mu^{2}}\mc{E}(\mu*\mc{N}_{\sigma})$:
{\small
\begin{align*}
 &\ \iint_{\mb{R}^{2d}}
 ( h(x)-\bar{\mu}(h))^{\top} \sigma^{2}(y)
 \int_{\mb{R}^{d}} \frac{\partial_{x}\varphi_{\sigma}(z-x)\varphi_{\sigma}(z-y)}
 {(\bar{\mu}*\varphi_{\sigma})(z)} \d z \d \bar{\mu}(x) \d \bar{\mu}(y) \\
 \leq &\ \left( \iiint_{\mb{R}^{3d}}
 \left|h(x)-\bar{\mu}(h)\right|^{2} \frac{|z-x|^{2}}{\sigma^{4}}
 \varphi_{\sigma}(z-x)
 \left| \sigma^{2}(y) \right|^{2} \varphi_{\sigma}(z-y) \d \bar{\mu}(x) \d \bar{\mu}(y)
 \d z \right)^{\frac{1}{2}}
 \leq C_{d}.
\end{align*}}Then, using a similar argument from \textbf{Step 5}, we determine that
\begin{equation*}
 \begin{aligned}
   \lim_{\alpha\ra0}\lim_{\beta\ra0} I^{4} = 0.
 \end{aligned}
\end{equation*}

\textbf{Step 7.} (Estimating $I^{5}$).
The estimate of $I^{5}$ is quite similar to those of $I^{3}$ and $I^{4}$.
We only present two essential inequalities, as was done in \textbf{Step 6}.
Estimating $\partial_{\mu\mu}^{2}\mc{G}(\mu,\nu)$ produces the following inequality:
\begin{align*}
 &\ \left| \iint_{\mb{R}^{2d}}
 \sigma^{2}(x)\sigma^{2,\top}(y)
 \partial_{x}\psi_{n}^{B}(x)\partial_{y}^{\top}\psi_{n}^{B}(y)
 \d \bar{\mu}(x) \d \bar{\mu}(y) \right. \\
 &\ \left. - \iint_{\mb{R}^{2d}} \sigma^{2}(x)\sigma^{2,\top}(y)
 \partial_{x}\psi_{n}^{B}(x)\partial_{y}^{\top}\psi_{n}^{B}(y)
 \d \ubar{\mu}(x) \d \ubar{\mu}(y) \right| \\
 \leq &\ C_{d} \int_{\Rd} \left| \sigma^{2}(x) \right| \partial_{x}\psi_{n}^{B}(x)
 \d \left( \bar{\mu} - \ubar{\mu} \right)(x) \lesssim_{d}\mc{G}(\mu,\nu) .
\end{align*}
Furthermore, an additional inequality arises from the estimate of
$\partial_{\mu\mu}^{2}\mc{E}(\mu*\mc{N}_{\sigma})$:
\begin{align*}
 &\ \iint_{\mb{R}^{2d}}
 \sigma^{2}(x)\sigma^{2,\top}(y)
 \int_{\mb{R}^{d}} \frac{\partial_{x}\varphi_{\sigma}(z-x)\partial_{y}^{\top}\varphi_{\sigma}(z-y)}
 {(\bar{\mu}*\varphi_{\sigma})(z)} \d z \d \bar{\mu}(x) \d \bar{\mu}(y) \\
 \leq &\  \iint_{\mb{R}^{2d}}
 \frac{|z-x|^{2}}{\sigma^{4}} \varphi_{\sigma}(z-x) \d \bar{\mu}(x) \d z
 \leq C_{d}.
\end{align*}
Applying a strategy analogous to the preceding two steps, we derive that
\begin{equation*}
 \begin{aligned}
   \lim_{\alpha\ra0}\lim_{\beta\ra0} I^{5} = 0.
 \end{aligned}
\end{equation*}

\textbf{Step 8.} (Deriving the contradiction).
Combining the estimates for $I^{1}, I^{2}, I^{3}, I^{4}$, and $I^{5}$,
we conclude that
\begin{equation*}
 \begin{aligned}
   \lim_{\alpha\ra0}\lim_{\beta\ra0} \left( I^{1}+I^{2}+I^{3}+I^{4}+I^{5} \right) = 0.
 \end{aligned}
\end{equation*}
This is contrary to \eqref{eq:contradiction}.
So far, we have completed the proof of the theorem.
\end{proof}

We already know that $u$ is a viscosity solution of \Cref{eq:HJB-1}.
In the meantime, let $\upsilon$ serve as an additional viscosity solution of \Cref{eq:HJB-1}.
In accordance with \Cref{thm:comparison}, $u\leq \upsilon$ and $\upsilon\leq u$.
Consequently, we instantly obtain $u\equiv\upsilon$.

\section{Application to partially observed diffusions}\label{sec:app}

We provide a brief explanation of the partially observed stochastic optimal control problem
and its corresponding separated problem.
On the probability space $(\Omega,\mc{F},\mb{P})$, we consider a $(d+q)$-dimensional
Wiener process $(W_{}^{1},W_{}^{2})$.
The state and observation equations are governed by stochastic differential equations
\begin{equation}\label{eq:state-observation}
 \left\{
  \begin{aligned}
    X_{t} =&\ X_0 + \int_{0}^{t} b(X_{r},\gamma_{r}) \d r
    + \int_{0}^{t} \sigma^{1}(X_{r},\gamma_{r}) \d W_{r}^{1}
    + \int_{0}^{t} \sigma^{2}(X_{r}) \d W_{r}^{2},\\
    Y_{t} =&\ \int_{0}^{t} h(X_{r}) \d r + W_{t}^{2},
  \end{aligned}
 \right.
\end{equation}
where $X_{0}\sim\mu\in\mc{P}(\Rd)$,
and the coefficients are all bounded and Lipschitz functions.
The control strategy $\gamma_{\cdot}$
will be explained in detail in \Cref{def:control}.
Note that $W^{2}$ is the correlated noise
between the state and the observation.

Our optimal control problem is to minimise
\begin{align}\label{eq:cost-2}
  \tilde{\mc{J}}(\mu;\gamma_{\cdot}) =&\ \mb{E}\left[ 
  \int_{0}^{\infty} e^{-t}
  \ell(X_{t},\gamma_{t}) \d t \right]
\end{align}
over all admissible controls,
where $\ell(\cdot,\cdot):\mb{R}^{d}\times\Gamma\ra\mb{R}^{1}$
is a 1-Lipschitz continuous function with respect to $x$, i.e.,
for all $\gamma\in\Gamma$,
\begin{align*}
  \mr{Lip}(\ell)
  := \sup_{x,y\in\Rd,~x\neq y}
  \frac{\left|\ell(x,\gamma) - \ell(y,\gamma) \right|}{|x-y|}
  \leq 1.
\end{align*}

We introduce the observation filtration
$\mc{F}_{t}^{Y}=\sigma\{ Y_{r},0\leq r\leq t \}$.
Let $\pi_{t}$ denote the optimal filter
$\mb{E}\left[ \phi(X_{t}) \left|\mc{F}_{t}^{Y} \right.\right]$,
where $\phi\in\mc{D}(\Rd)$.
By applying the \emph{reference probability method}
or the \emph{innovations method} in nonlinear filtering,
we derive the K-S equation satisfied
by the $\mc{P}(\mathbb{R}^d)$-valued process $\pi_t$ as follows:
 \begin{align}\label{eq:Kushner-innov}
  \pi_{t}(\phi) = \pi_{0}(\phi)
  + \int_{0}^{t} \pi_{r}(\msc{L}^{\gamma}\phi) \d r
  + \int_{0}^{t}
  \bra \pi_{r}(h\phi + \msc{M}^{}\phi) - \pi_{r}(h)\pi_{r}(\phi),\d I_{r} \ket_{\mb{R}^{q}},
\end{align}
where the innovation
$I_{t} := Y_{t} - \int_{0}^{t}\pi_{r}(h) \d r$ is an $\mc{F}_{t}^{Y}$-Wiener process.
Due to the \emph{circular dependency} between control and innovation filtration
$\mc{F}_{t}^{I}=\sigma\{ I_{r},0\leq r\leq t \}$,
it is necessary to carefully define the \emph{weak admissibility} of control strategies.

\begin{definition}\label{def:control}
A $6$-tuple $(\bar{\Omega},\mc{G},\{\mc{G}_{t}^{}\}_{t\geq 0},\bar{\mb{P}},\beta_{t},\gamma_{t})$
is called a weak admissible control,
and $(\pi_{t}, \gamma_{t},\beta_{t})$ a weak admissible system,
if the following conditions hold:
\begin{enumerate}[\rm \bfseries(1)]
\item
$(\bar{\Omega},\mc{G},\{\mc{G}_{t}^{}\}_{t\geq 0},\bar{\mb{P}})$
is a filtered probability space satisfying the usual conditions;
\item
$\beta_{t}$ is a $q$-dimensional Wiener process defined
on $(\bar{\Omega},\mc{G},\{\mc{G}_{t}^{}\}_{t\geq 0},\bar{\mb{P}})$;
\item
$\gamma_{t}$ is an $\{\mc{G}_{t}\}_{t\geq 0}$-adapted process on
$(\bar{\Omega},\mc{G},\{\mc{G}_{t}^{}\}_{t\geq 0},\bar{\mb{P}})$ taking values in $\Gamma$;
\item
$\pi_{t}$ is the unique strong solution of \Cref{eq:Kushner} 
on $(\bar{\Omega},\mc{G},\{\mc{G}_{t}^{}\}_{t\geq 0},\bar{\mb{P}})$ under $\gamma_{t}$.
\end{enumerate}
We denote by $\bm{\Gamma}^{w}$ the set of weak admissible control processes.
\end{definition}
Two points in \Cref{def:control} warrant attention:
First, we employ \Cref{eq:Kushner} rather than \Cref{eq:Kushner-innov};
second, the filtration $\mc{G}_{t}$ is neither
$\mc{F}_{t}^{Y}$ nor $\mc{F}_{t}^{I}$,
as it is solely used to define the adapted solution of \Cref{eq:Kushner}.

One can rewrite the functional $\tilde{\mc{J}}$ in \Cref{eq:cost-2}
in terms of $\pi_{t}$, i.e.,
\begin{align*}
  \tilde{\mc{J}}(\mu;\gamma_{\cdot}) =&\ \int_{0}^{\infty}
  \mb{E}^{} \left[ e^{-t} \ell(X_{t},\gamma_{t}) \right] \d t
  = \mb{E}^{} \left[ \int_{0}^{\infty}
  \mb{E}^{} \left[ e^{-t} \ell(X_{t},\gamma_{t})
  \left|\mc{G}_{t}^{} \right.\right] \d t \right] \\
  =&\ \mb{E}^{} \left[ \int_{0}^{\infty}
  e^{-t} \pi_{t}(\ell(X_{t},\gamma_{t})) \d t \right]
  \triangleq \mb{E}^{} \left[ \int_{0}^{\infty}
  e^{-t} L(\pi_{t},\gamma_{t}) \d t \right]
  =: \mc{J}(\mu;\gamma_{\cdot}).
\end{align*}

The separated control problem in its weak formulation
is stated as follows:
Given $\pi_{0}=\mu$ and $\ell$ (i.e., $L$),
find a weak admissible system $(\pi_{t},\gamma_{t},\beta_{t})$
that minimizes $\mc{J}(\mu;\gamma_{\cdot})$ over $\bm{\Gamma}^{w}$.
We emphasize that the separated problem
is defined independently of the partially observed problem,
as described by \Cref{eq:state-observation,eq:cost-2}.
Under more restrictive assumptions,
it can be shown that the value functions defined by
\Cref{eq:cost-1,eq:cost-2} are identical.
However, these results fall outside the scope of this paper;
a comprehensive analysis of this issue is provided in \cite{ElKDJP1988}.
We intend to explore related topics in future research.
Once the weak formulation of the separated problem is obtained,
the previously derived results can be leveraged to help solve
the separated problem under the strong formulation
in a fixed filtered probability space $(\Omega,\mc{F},\{\mc{F}_{t}^{}\}_{t\geq 0},\mb{P})$.
Here, we observe that \Cref{hyp:BC} is automatically satisfied,
while \Cref{hyp:Lip1} follows directly from
the Kantorovich-Rubinstein duality (cf. \cite[Lecture 3]{ABS2021}).


\section{Conclusion} 

This paper investigates a large class of separated problems modeled by the K-S equation.
The DPP for this type of optimal control problem
gives birth to a novel category of HJB equations.
The HJB equation includes both variational derivatives and Lions derivatives.
We established the comparison theorem of viscosity solutions
by altering the smooth gauge-type function developed in \cite{CGKPR2024}
and employing the entropy functional penalization in \cite{DS2024}.
The relatively weak definitions of variational derivatives and Lions derivatives
in the present work make the concept of viscosity solutions
proposed in this paper a general and feasible framework.
Moreover, the results presented in this paper lay a solid groundwork
for further analysis of the separated problem,
such as examining the verification theorem associated with the HJB equation
(i.e., the optimality conditions)
as well as the existence of optimal feedback control, among other considerations.

We emphasize that the correlated noise coefficient $\sigma^{2}(\cdot)$ can be \emph{control-dependent},
and the methodology provided in this paper can fully handle this situation.

Thus far, we have only been able to tackle the situation
in which the coefficients of state-observation pairs
are \emph{bounded} and \emph{Lipschitz continuous}.
When the coefficients of the state equation \emph{grow linearly},
the authors believe that an infinite-dimensional version of the Crandall-Ishii lemma is inevitable.





\bibliography{references}

\end{document}